\documentclass[11pt]{amsart}
\usepackage{hyperref}
\usepackage{amsfonts}
\usepackage{amsmath}
\usepackage{amssymb}
\usepackage{amscd}
\usepackage{graphicx}
\usepackage{enumitem}
\usepackage{latexsym}
\usepackage{color}
\usepackage{epsfig}
\usepackage{amsopn}
\usepackage{xypic}
\usepackage{mathrsfs}
\usepackage{array}

\usepackage{booktabs}
\usepackage{longtable}
\theoremstyle{plain}
\newtheorem{lemma}{Lemma}[section]
\newtheorem*{theorem*}{Theorem}
\newtheorem*{lemma*}{Lemma}
\newtheorem*{proposition*}{Proposition}
\newtheorem*{conjecture*}{Conjecture}
\newtheorem*{corollary*}{Corollary}
\newtheorem*{problem*}{Problem}
\newtheorem{theorem}[lemma]{Theorem}
\newtheorem{theorem-definition}[lemma]{Theorem and Definition}
\newtheorem{conjecture}[lemma]{Conjecture}
\newtheorem{corollary}[lemma]{Corollary}
\newtheorem{proposition}[lemma]{Proposition}

\newtheorem{problem}[lemma]{Problem}
\newtheorem{question}[lemma]{Question}

\theoremstyle{definition}
\newtheorem{definition}[lemma]{Definition}
\newtheorem{example}[lemma]{Example}
\newtheorem{remark}[lemma]{Remark}
\newtheorem{exercise}[lemma]{Exercise}

\newcommand{\F}[1]{\mathscr{#1}}
\newcommand{\fto}[1]{\stackrel{#1}{\to}}
\newcommand{\OV}[1]{\overline{#1}}

\newcommand{\Z}{\mathbb{Z}}

\renewcommand{\F}{\mathbb{F}}
\newcommand{\C}{\mathbb{C}}
\newcommand{\Q}{\mathbb{Q}}
\newcommand{\R}{\mathbb{R}}

\newcommand{\OO}{\mathcal{O}}
\newcommand{\te}{\otimes}

\newcommand{\cI}{\mathcal{I}}
\newcommand{\cF}{\mathcal F}
\newcommand{\cC}{\mathcal C}

\newcommand{\cA}{\mathcal A}
\newcommand{\cM}{\mathcal M}
\newcommand{\alg}{\mathrm{alg}}
\newcommand{\cU}{\mathcal U}

\newcommand{\cP}{\mathcal P}
\newcommand{\cT}{\mathcal{T}}
\newcommand{\cE}{\mathcal{E}}

\newcommand{\rH}{\mbox{H}}
\newcommand{\gr}{\mathrm{gr}}
\newcommand{\cZ}{\mathcal{Z}}

\newcommand{\HH}{\mathbb{H}}

\newcommand{\leqpar}{\underset{{\scriptscriptstyle (}-{\scriptscriptstyle )}}{<}}
\newcommand{\geqpar}{\underset{{\scriptscriptstyle (}-{\scriptscriptstyle )}}{>}}

\let\tilde\widetilde
\let\hat\widehat

\renewcommand{\P}{\mathbb{P}}

\DeclareMathOperator{\Bl}{Bl}
\DeclareMathOperator{\BBs}{\mathbb{B}s}
\DeclareMathOperator{\ch}{ch}

\DeclareMathOperator{\Gr}{Gr}
\newcommand{\tr}{\mathrm{tr}}
\DeclareMathOperator{\Bs}{Bs}
\DeclareMathOperator{\Hom}{Hom}

\DeclareMathOperator{\Pic}{Pic}
\DeclareMathOperator{\cone}{cone}

\DeclareMathOperator{\Sym}{Sym}
\DeclareMathOperator{\Spec}{Spec}
\DeclareMathOperator{\Bigc}{Big}
\DeclareMathOperator{\NS}{NS}
\DeclareMathOperator{\im}{Im}
\DeclareMathOperator{\td}{td}

\newcommand{\Sch}{\mathrm{Sch}}
\DeclareMathOperator{\Set}{Set}

\DeclareMathOperator{\rk}{rk}

\DeclareMathOperator{\Mor}{Mor}
\DeclareMathOperator{\Ext}{Ext}
\DeclareMathOperator{\ext}{ext}

\DeclareMathOperator{\id}{id}

\DeclareMathOperator{\sHom}{\mathcal{H} \textit{om}}
\DeclareMathOperator{\sEnd}{\mathcal{E} \textit{nd}}

\DeclareMathOperator{\coh}{coh}
\DeclareMathOperator{\Stab}{Stab}
\DeclareMathOperator{\Amp}{Amp}
\DeclareMathOperator{\Eff}{Eff}
\DeclareMathOperator{\Nef}{Nef}

\DeclareMathOperator{\num}{num}

\begin{document}

\date{\today}
\author[J. Huizenga]{Jack Huizenga}
\address{Department of Mathematics, The Pennsylvania State University, University Park, PA 16802}
\email{huizenga@psu.edu}
\subjclass[2010]{Primary: 14J60. Secondary: 14E30, 14J29, 14C05}
\keywords{Moduli spaces of sheaves, Hilbert schemes of points, ample cone, Bridgeland stability}
\thanks{During the preparation of this article the author was partially supported by a National Science Foundation Mathematical Sciences Postdoctoral Research Fellowship}

\title{Birational geometry of moduli spaces of sheaves and Bridgeland stability}

\begin{abstract}
Moduli spaces of sheaves and Hilbert schemes of points have experienced a recent resurgence in interest in the past several years, due largely to new techniques arising from Bridgeland stability conditions and derived category methods.  In particular, classical questions about the birational geometry of these spaces can be answered by using new tools such as the positivity lemma of Bayer and Macr\`i.  In this article we first survey classical results on moduli spaces of sheaves and their birational geometry.  We then discuss the relationship between these classical results and the new techniques coming from Bridgeland stability, and discuss how cones of ample divisors on these spaces can be computed with these new methods.  This survey expands upon the author's talk at the 2015 Bootcamp in Algebraic Geometry preceding the 2015 AMS Summer Research Institute on Algebraic Geometry at the University of Utah.
\end{abstract}
\maketitle

\setcounter{tocdepth}{1}
\tableofcontents

\section{Introduction}

The topic of vector bundles in algebraic geometry is a broad field with a rich history.  In the 70's and 80's, one of the main questions of interest was the study of low rank vector bundles on projective spaces $\P^r$.  One particularly challenging conjecture in this subject is the following.  \begin{conjecture}[Hartshorne \cite{Hartshorne}]If $r\geq 7$ then any rank $2$ bundle on $\P^r_\C$ splits as a direct sum of line bundles.   
\end{conjecture}
The Hartshorne conjecture is connected to the study of subvarieties of projective space of small codimension.  In particular, the above statement implies that if $X\subset \P^r$ is a codimension $2$ smooth subvariety and $K_X$ is a multiple of the hyperplane class then $X$ is a complete intersection.  Thus, early intersect in the study of vector bundles was born out of classical questions in projective geometry.

Study of these types of questions led naturally to the study of \emph{moduli spaces} of (semistable) vector bundles, parameterizing the isomorphism classes of (semistable) vector bundles with given numerical invariants on a projective variety $X$ (we will define \emph{semistable} later---for now, view it as a necessary condition to get a good moduli space).  As often happens in mathematics, these spaces have become interesting in their own right, and their study has become an entire industry.  Beginning in the 80's and 90's, and continuing to today, people have studied the basic questions of the geometry of these spaces.  Are they smooth? Irreducible?  What do their singularities look like?  When is the moduli space nonempty?  What are divisors on the moduli space?  Especially when $X$ is a curve or surface, satisfactory answers to these questions can often be given.  We will survey several foundational results of this type in \S\ref{sec-moduli}-\ref{sec-properties}.

More recently, there has been a great deal of interest in the study of the \emph{birational geometry} of moduli spaces of various geometric objects.  Loosely speaking, the goal of such a program is to understand alternate birational models, or \emph{compactifications}, of a moduli space as themselves being moduli spaces for slightly different geometric objects.  For instance, the Hassett-Keel program \cite{HassettHyeon} studies alternate compactifications of the Deligne-Mumford compactification $\overline M_g$ of the moduli space of stable curves.  Different compactifications can be obtained by studying (potentially unstable) curves with different types of singularities.  In addition to being interesting in their own right, moduli spaces provide explicit examples of higher dimensional varieties which can frequently be understood in great detail.  We survey the birational geometry of moduli spaces of sheaves from a classical viewpoint in \S\ref{sec-classical}.

In the last several years, there has been a great deal of progress in the study of the birational geometry of moduli spaces of sheaves owing to Bridgeland's introduction of the concept of a \emph{stability condition} \cite{bridgeland:stable,bridgelandK3}.  Very roughly, there is a complex manifold $\Stab(X)$, the \emph{stability manifold}, parameterizing stability conditions $\sigma$ on $X$.  There is a moduli space corresponding to each condition $\sigma$, and the stability manifold decomposes into chambers  where the corresponding moduli space does not change as $\sigma$ varies in the chamber.  For one of these chambers, the \emph{Gieseker chamber}, the corresponding moduli space is the ordinary moduli space of semistable sheaves.  The moduli spaces corresponding to other chambers often happen to be the alternate birational models of the ordinary moduli space.  In this way, the birational geometry of a moduli space of sheaves can be viewed in terms of a variation of the moduli problem.  In \S\ref{sec-Bridgeland} we will introduce Bridgeland stability conditions, and especially study stability conditions on a surface.  We study some basic examples on $\P^2$ in \S\ref{sec-exP2}.  Finally, we close the paper in \S\ref{sec-positivity} by surveying some recent results on the computation of ample cones of Hilbert schemes of points and moduli spaces of sheaves on surfaces.

\subsection*{Acknowledgements}  I would especially like to thank Izzet Coskun and Benjamin Schmidt for many discussions on Bridgeland stability and related topics.  In addition, I would like to thank the referee of this article for many valuable comments, as well as Barbara Bolognese, Yinbang Lin, Eric Riedl, Matthew Woolf, and Xialoei Zhao.  Finally, I would like to thank the organizers of the 2015 Bootcamp in Algebraic Geometry and the 2015 AMS Summer Research Institute on Algebraic Geometry, as well as the funding organizations for these wonderful events.

\section{Moduli spaces of sheaves}\label{sec-moduli}

The definition of a Bridgeland stability condition is motivated by the classical theory of semistable sheaves.  In this section we review the basics of the theory of moduli spaces of sheaves, particularly focusing on the case of a surface.  The standard references for this material are Huybrechts-Lehn \cite{HuybrechtsLehn} and Le Potier \cite{LePotierLectures}.

\subsection{The moduli property}  First we state an idealized version of the moduli problem.  Let $X$ be a smooth projective variety with polarization $H$, and fix a set of discrete numerical invariants of a coherent sheaf $E$ on $X$.  This can be accomplished by fixing the \emph{Hilbert polynomial} $$P_E(m) = \chi(E\te \OO_X(mH))$$ of the sheaf.  

A \emph{family of sheaves on $X$ over $S$} is a (coherent) sheaf $\cE$ on $X\times S$ which is $S$-flat.  For a point $s\in S$, we write $E_s$ for the sheaf $\cE|_{X\times \{s\}}$.   We say $\cE$ is a family of \emph{semistable sheaves of Hilbert polynomial $P$} if $E_s$ is semistable with Hilbert polynomial $P$ for each $s\in S$ (see \S\ref{ssec-semistable} for the definition of semistability).  We define a \emph{moduli functor} $$\cM'(P):\Sch^{o}\to \Set$$ by defining $\cM'(P)(S)$ to be the set of isomorphism classes of families of semistable sheaves on $X$ with Hilbert polynomial $P$.  We will sometimes just write $\cM'$ for the moduli functor when the polynomial $P$ is understood.

Let $p:X\times S\to S$ be the projection.  If $\cE$ is a family of semistable sheaves on $X$ with Hilbert polynomial $P$ and $L$ is a line bundle on $S$, then $\cE\te p^*L$ is again such a family.  The sheaves $E_s$ and $(\cE\te p^*L)|_{X\times \{s\}}$ parameterized by any point $s\in S$ are isomorphic, although $\cE$ and $\cE\te p^*L$ need not be isomorphic.  We call two  families of sheaves on $X$ \emph{equivalent} if they differ by tensoring by a line bundle pulled back from the base, and define a refined moduli functor $\cM$ by modding out by this equivalence relation: $\cM= \cM'/\sim$.

The basic question is whether or not $\cM$ can be represented by some nice object, e.g. by a projective variety or a scheme.  We recall the following definitions.

\begin{definition}
A functor $\cF:\Sch^o\to \Set$ is \emph{represented} by a scheme $X$ if there is an isomorphism of functors $\cF \cong \Mor_{\Sch}(-,X)$.

A functor $\cF: \Sch^o\to \Set$ is \emph{corepresented} by a scheme $X$ if there is a natural transformation $\alpha:\cF\to \Mor_{\Sch}(-,X)$ with the following universal property:  if $X'$ is a scheme and $\beta:\cF\to \Mor_{\Sch}(-,X')$ a natural transformation, then there is a unique morphism $\pi: X\to X'$ such that $\beta$ is the composition of $\alpha$ with the transformation $\Mor_{\Sch}(-,X)\to \Mor_{\Sch}(-,X')$ induced by $\pi$.
\end{definition}

\begin{remark}
Note that if $\cF$ is represented by $X$ then it is also corepresented by $X$.  

If $\cF$ is represented by $X$, then $\cF(\Spec\C) \cong \Mor_{\Sch}(\Spec \C,X)$.  That is, the points of $X$ are in bijective correspondence with $\cF(\Spec \C)$.  This need not be true if $\cF$ is only corepresented by $X$.

If $\cF$ is corepresented by $X$, then $X$ is unique up to a unique isomorphism.
\end{remark}

We now come to the basic definition of moduli space of sheaves.

\begin{definition}
A scheme $M(P)$ is a \emph{moduli space of semistable sheaves with Hilbert polynomial $P$} if $M(P)$ corepresents $\cM(P)$.  It is a \emph{fine moduli space} if it represents $\cM(P)$.
\end{definition}

The most immediate consequence of $M$ being a moduli space is the existence of the \emph{moduli map}.  Suppose $E$ is a family of semistable sheaves on $X$ parameterized by $S$.  Then we obtain a morphism $S\to M$ which intuitively sends $s\in S$ to the isomorphism class of the sheaf $E_s$.  

In the special case when the base $\{s\}$ is a point, a family over $\{s\}$ is the isomorphism class of a single sheaf, and the moduli map $\{s\}\to M$ sends that class to a corresponding point.  The compatibilities in the definition of a natural transformation ensure that in the case of a family $\cE$ parameterized by a base $S$ the image in $M$ of a point $s\in S$ depends only on the isomorphism class of the sheaf $E_s$ parameterized by $s$.  

In the ideal case where the moduli functor $\cM$ has a fine moduli space, there is a universal sheaf $\cU$ on $X$ parameterized by $M$.  We have an isomorphism $$\cM(M)\cong \Mor_\Sch(M,M)$$ and the distinguished identity morphism $M\to M$ corresponds to a family $\cU$ of sheaves parameterized by $M$ (strictly speaking, $\cU$ is only well-defined up to tensoring by a line bundle pulled back from $M$).  This universal sheaf has the property that if $\cE$ is a family of semistable sheaves on $X$ parameterized by $S$ and $f:S\to M$ is the moduli map, then $\cE$ and $({\mathrm{id}}_X\times f)^*\cU$ are equivalent. 

\subsection{Issues with a naive moduli functor}  In this subsection we give some examples to illustrate the importance of the as-yet-undefined \emph{semistability} hypothesis in the definition of the moduli functor.  Let $\cM^n$ be the \emph{naive} moduli functor of (flat) families of coherent sheaves with Hilbert polynomial $P$ on $X$, omitting any semistability hypothesis.  We might hope that this functor is (co)representable by a scheme $M^n$ with some nice properties, such as the following.
\begin{enumerate}
\item $M^n$ is a scheme of finite type.
\item The points of $M^n$ are in bijective correspondence with isomorphism classes of coherent sheaves on $X$ with Hilbert polynomial $P$.
\item A family of sheaves over a smooth punctured curve $C-\{\mathrm{pt}\}$ can be uniquely completed to a family of sheaves over $C$.
\end{enumerate}

However, unless some restrictions are imposed on the types of sheaves which are allowed, all three hopes will fail.  Properties (2) and (3) will also typically fail for semistable sheaves, but this failure occurs in a well-controlled way. 

\begin{example}
Consider $X=\P^1$, and let $P=P_{\OO_{\P^1}^{\oplus 2}} = 2m+2$ be the Hilbert polynomial of the rank $2$ trivial bundle.  Then for any $n\geq 0$, the bundle $$\OO_{\P^1}(n)\oplus \OO_{\P^1}(-n)$$ also has Hilbert polynomial $2m+2$, and $h^0(\OO_{\P^1}(n)\oplus \OO_{\P^1}(-n))=n+1$.  If there is a moduli scheme $M^n$ parameterizing all sheaves on $\P^1$ of Hilbert polynomial $P$, then $M^n$ cannot be of finite type.  Indeed, the loci $$W_n = \{E : h^0(E)\geq n\} \subset M(P)$$ would then form an infinite decreasing chain of closed subschemes of $M(P)$.
\end{example}

\begin{example}\label{ex-Sequiv1}
Again consider $X=\P^1$ and $P=2m+2$.  Let $S = \Ext^1(\OO_{\P^1}(1),\OO_{\P^1}(-1))=\C$.  For $s\in S$, let $E_s$ be the sheaf $$0\to \OO_{\P^1}(-1)\to E_s \to \OO_{\P^1}(1)\to 0$$ defined by the extension class $s$.  One checks that if $s\neq 0$ then $E_s \cong \OO_{\P^1}^{\oplus 2}$, but the extension is split for $s=0$. It follows that the moduli map $S\to M^n$ must be constant, so $\OO_{\P^1}\oplus \OO_{\P^1}$ and $\OO_{\P^1}(1)\oplus \OO_{\P^1}(-1)$ are identified in the moduli space $M^n$.
\end{example}

\begin{example}\label{ex-Sequiv}
Suppose $X$ is a smooth variety and $F$ is a coherent sheaf with $\dim\Ext^1(F,F)\geq 1$.  Let $S\subset \Ext^1(F,F)$ be a $1$-dimensional subspace, and for any $s\in S$ let $E_s$ be the corresponding extension of $F$ by $F$.  Then if $s,s'\in S$ are both not zero, we have $$E_s\cong E_{s'} \not\cong E_0 = F\oplus F.$$  As in the previous example, we see that $F\oplus F$ and a nontrivial extension of $F$ by $F$ must be identified in $M^n$.  Therefore any two extensions of $F$ by $F$ must also be identified in $M^n$.
\end{example}

If $F$ is semistable, then Example  \ref{ex-Sequiv} is an example of a nontrivial family of \emph{$S$-equivalent} sheaves.  A major theme of this survey is that $S$-equivalence is the main source of interesting birational maps between moduli spaces of sheaves.

\subsection{Semistability}\label{ssec-semistable} Let $E$ be a coherent sheaf on $X$.  We say that $E$ is \emph{pure} of dimension $d$ if the support of $E$ is $d$-dimensional and every nonzero subsheaf of $E$ has $d$-dimensional support.  

\begin{remark} If $\dim X=n$, then  $E$ is pure of dimension $n$ if and only if $E$ is torsion-free.  \end{remark}

If $E$ is pure of dimension $d$ then the Hilbert polynomial $P_E(m)$ has degree $d$.  We write it in the form $$P_E(m)=\alpha_d(E) \frac{m^d}{d!}+\cdots,$$ and define the \emph{reduced Hilbert polynomial} by $$p_E(m) = \frac{P_E(m)}{\alpha_d(E)}.$$ In the principal case of interest where $d=n=\dim X$, Riemann-Roch gives $\alpha_n(E) = r(E)H^n$ where $r(E)$ is the rank, and $$p_E(m) = \frac{P_E(m)}{r(E)H^n}.$$ 

\begin{definition}
A sheaf $E$ is \emph{(semi)stable} if it is pure of dimension $d$ and any proper subsheaf $F\subset E$ has $$p_F \leqpar p_E,$$ where polynomials are compared at large values.  That is, $p_F<p_E$ means that $p_F(m)<p_E(m)$ for all $m\gg 0$.
\end{definition}

The above notion of stability is often called Gieseker stability, especially when a distinction from other forms of stability is needed.  The foundational result in this theory is that Gieseker semistability is the correct extra condition on sheaves to give well-behaved moduli spaces.

\begin{theorem}[{\cite[Theorem 4.3.4]{HuybrechtsLehn}}]
Let $(X,H)$ be a smooth, polarized projective variety, and fix a Hilbert polynomial $P$.  There is a projective moduli scheme of semistable sheaves on $X$ of Hilbert polynomial $P$.
\end{theorem}

While the definition of Gieseker stability is compact, it is frequently useful to use the Riemann-Roch theorem to make it more explicit.  We spell  this out in the case of a curve or surface.  We define the \emph{slope} of a coherent sheaf $E$ of positive rank on an $n$-dimensional variety by $$\mu(E) = \frac{c_1(E).H^{n-1}}{r(E)H^n}.$$ 
\begin{example}[Stability on a curve]
Suppose $C$ is a smooth curve of genus $g$.  The Riemann-Roch theorem asserts that if $E$ is a coherent sheaf on $C$ then $$\chi(E) = c_1(E)+ r(E)(1-g).$$ Polarizing $C$ with $H=p$ a point, we find $$P_E(m)=\chi(E(m)) = c_1(E(m))+r(E)(1-g)=r(E)m +(c_1(E)+r(E)(1-g)),$$ and so $$p_E(m) = m+\frac{c_1(E)}{r(E)}+(1-g).$$ We conclude that if $F\subset E$ then $p_F \leqpar p_E$ if and only if $\mu(F)\leqpar \mu(E)$.
\end{example}

\begin{example}[Stability on a surface]\label{ex-stabilitySurface}
Let $X$ be a smooth surface with polarization $H$, and let $E$ be a sheaf of positive rank.  We define the \emph{total slope} and \emph{discriminant} by $$\nu(E) = \frac{c_1(E)}{r(E)}\in H^2(X,\Q) \qquad \textrm{and} \qquad \Delta(E) = \frac{1}{2}\nu(E)^2-\frac{\ch_2(E)}{r}\in \Q.$$ With this notation, the Riemann-Roch theorem takes the particularly simple form $$\chi(E) = r(E)(P(\nu(E))-\Delta(E)),$$ where $P(\nu)= \chi(\OO_X)+\frac{1}{2}\nu(\nu-K_X)$ (see \cite{LePotierLectures}).  The total slope and discriminant behave well with respect to tensor products: if $E$ and $F$ are locally free then \begin{align*}\nu(E\te F) &= \nu(E)+\nu(F)\\ \Delta(E\te F)&= \Delta(E) + \Delta(F).\end{align*} Furthermore, $\Delta(L) = 0$ for a line bundle $L$; equivalently, in the case of a line bundle the Riemann-Roch formula is $\chi(L) = P(c_1(L))$.  Then we compute
\begin{align*}
\chi(E(m)) &= r(E)(P(\nu(E)+mH)-\Delta(E))\\
&= r(E)(\chi(\OO_X)+\frac{1}{2}(\nu(E)+mH)(\nu(E)+mH-K_X)-\Delta(E))\\
&= r(E)(P(\nu(E))+\frac{1}{2}(mH)^2+mH.(\nu(E)+\frac{1}{2}K_X)-\Delta(E))\\
&= \frac{r(E)H^2}{2}m^2+r(E)H.(\nu(E)+\frac{1}{2}K_X)m + \chi(E),
\end{align*}
so $$p_E(m) = \frac{1}{2}m^2+\frac{H.(\nu(E)+\frac{1}{2}K_X)}{H^2}m+\frac{\chi(E)}{r(E)H^2}$$
Now if $F\subset E$, we compare the coefficients of $p_F$ and $p_E$ lexicographically to determine when $p_F \leqpar p_E$.  We see that $p_F \leqpar p_E$ if and only if either $\mu(F) < \mu(E)$, or $\mu(F)=\mu(E)$ and $$\frac{\chi(F)}{r(F)H^2}\leqpar \frac{\chi(E)}{r(E)H^2}.$$
\end{example}

\begin{example}[Slope stability]\label{ex-slopestability}
The notion of \emph{slope semistability} has also been studied extensively and frequently arises in the study of Gieseker stability.  We say that a torsion-free sheaf $E$ on a variety $X$ with polarization $H$ is \emph{$\mu$-(semi)stable} if every subsheaf $F\subset E$ of strictly smaller rank has $\mu(F)\leqpar \mu(E)$.  As we have seen in the curve and surface case, the coefficient of $m^{n-1}$ in the reduced Hilbert polynomial $p_E(m)$ is just $\mu(E)$ up to adding a constant depending only on $(X,H)$.  This observation gives the following chain of implications:
$$\mu\textrm{-stable} \Rightarrow \textrm{stable} \Rightarrow \textrm{semistable} \Rightarrow \mu\textrm{-semistable}.$$  While Gieseker (semi)stability gives the best moduli theory and is therefore the most common to work with, it is often necessary to consider these various other forms of stability to study ordinary stability.
\end{example}

\begin{example}[Elementary modifications]\label{ex-elementary}
As an example where $\mu$-stability is useful, suppose $X$ is a smooth surface and $E$ is a torsion-free sheaf on $X$.  Let $p\in X$ be a point where $X$ is locally free, and consider sheaves $E'$ defined as kernels of maps $E\to \OO_p$, where $\OO_p$ is a skyscraper sheaf:  $$0\to E'\to E\to \OO_p\to 0.$$  Intuitively, $E'$ is just $E$ with an additional simple singularity imposed at $p$.  Such a sheaf $E'$ is called an \emph{elementary modification} of $E$.  We have $\mu(E)=\mu(E')$ and $\chi(E') = \chi(E)-1$, which makes elementary modifications a useful tool for studying sheaves by induction on the Euler characteristic.

Suppose $E$ satisfies one of the four types of stability discussed in Example \ref{ex-slopestability}.  If $E$ is $\mu$-(semi)stable, then it follows that $E'$ is $\mu$-(semi)stable as well.  Indeed, if $F\subset E'$ with $r(F)<r(E')$, then also $F\subset E$, so $\mu(F)\leqpar\mu(E)$.  But $\mu(E)=\mu(E')$, so $\mu(F)\leqpar\mu(E')$ and $E'$ is $\mu$-(semi)stable.

On the other hand, elementary modifications do not behave as well with respect to Gieseker (semi)stability.  For example, take $X=\P^2$.  Then $E= \OO_{\P^2}\oplus \OO_{\P^2}$ is semistable, but any any elementary modification $E'$ of $\OO_{\P^2}\oplus\OO_{\P^2}$ at a point $p\in \P^2$ is isomorphic to $I_p\oplus \OO_{\P^2}$, where $I_p$ is the ideal sheaf of $p$.  Thus $E'$ is not semistable.

It is also possible to give an example of a stable sheaf $E$ such that some elementary modification is not stable.  Let $p,q,r\in \P^2$ be distinct points. Then $\ext^1(I_r,I_{\{p,q\}})=2$.  If $E$ is any non-split extension $$0\to I_{\{p,q\}}\to E\to I_r\to 0$$ then $E$ is clearly $\mu$-semistable.  In fact, $E$ is stable: the only stable sheaves $F$ of rank $1$ and slope $0$ with $p_F\leq p_E$ are $\OO_{\P^2}$ and $I_s$ for $s\in \P^2$ a point, but $\Hom(I_s,E)=0$ for any $s\in \P^2$ since the sequence is not split.  Now if $s\in \P^2$ is a point distinct from $p,q,r$ and $E\to \OO_s$ is a map such that the composition $I_{\{p,q\}}\to E\to \OO_s$ is zero, then the corresponding elementary modification $$0\to E'\to E\to \OO_s\to 0$$ has a subsheaf $I_{\{p,q\}}\subset E'$.  We have $p_{I_{\{p,q\}}} = p_{E'}$, so $E'$ is strictly semistable.
\end{example}

\begin{example}[Chern classes]\label{ex-Chern}
Let $K_0(X)$ be the Grothendieck group of $X$, generated by classes $[E]$ of locally free sheaves, modulo relations $[E] = [F]+[G]$ for every exact sequence $$0\to F\to E\to G\to 0.$$ There is a symmetric bilinear \emph{Euler pairing} on $K_0(X)$ such that $([E],[F]) = \chi(E\te F)$ whenever $E,F$ are locally free sheaves.  The \emph{numerical Grothendieck group} $K_{\num}(X)$ is the quotient of $K_0(X)$ by the kernel of the Euler pairing, so that the Euler pairing descends to a nondegenerate pairing on $K_{\num}(X)$.

It is often preferable to fix the Chern classes of a sheaf instead of the Hilbert polynomial.  This is accomplished by fixing a class ${\bf v}\in K_{\num}(X)$.
Any class ${\bf v}$ determines a Hilbert polynomial $P_{\bf v} = ({\bf v},[\OO_X(m)])$.  In general, a polynomial $P$ can arise as the Hilbert polynomial of several classes ${\bf v}\in K_{\num}(X)$.  In any family $\cE$ of sheaves parameterized by a connected base $S$ the sheaves $E_s$ all have the same class in $K_{\num}(X)$.  Therefore, the moduli space $M(P)$ splits into connected components corresponding to the different vectors ${\bf v}$ with $P_{\bf v}= P$.  We write $M({\bf v})$ for the connected component of $M(P)$ corresponding to ${\bf v}$.
\end{example}

\subsection{Filtrations} In addition to controlling subsheaves, stability also restricts the types of maps that can occur between sheaves.

\begin{proposition}\label{prop-seesaw}
\begin{enumerate}
\item (See-saw property) In any exact sequence of pure sheaves $$0\to F\to E\to Q\to 0$$ of the same dimension $d$, we have $p_F \leqpar p_E$ if and only if $p_E \leqpar p_Q$.

\item If $F,E$ are semistable sheaves of the same dimension $d$ and $p_F > p_E$, then $\Hom(F,E)=0$.  

\item If $F,E$ are stable sheaves and $p_F=p_E$, then any nonzero homomorphism $F\to E$ is an isomorphism.

\item Stable sheaves $E$ are \emph{simple:} $\Hom(E,E)=\C$.
\end{enumerate}
\end{proposition}
\begin{proof}
(1)  We have $P_E = P_F + P_Q$, so $\alpha_d(E) = \alpha_d(F)+\alpha_d(Q)$ and $$p_E = \frac{P_E}{\alpha_d(E)} = \frac{P_F+P_Q}{\alpha_d(E)} = \frac{\alpha_d(F)p_F+\alpha_d(Q)p_Q}{\alpha_d(E)}.$$ Thus $p_E$ is a weighted mean of $p_F$ and $p_Q$, and the result follows.

(2) Let $f:F\to E$ be a homomorphism, and put $C = \im f$ and $K = \ker f$.  Then $C$ is pure of dimension $d$ since $E$ is, and $K$ is pure of dimension $d$ since $F$ is.  By (1) and the semistability of $F$, we have $p_C\geq p_F > p_E$.  This contradicts the semistability of $E$ since $C\subset E$.

(3) Since $p_F = p_E$, $F$ and $E$ have the same dimension.  With the same notation as in (2), we instead find $p_C\geq p_F = p_E$, and the stability of $E$ gives $p_C=p_E$ and $C=E$.  If $f$ is not an isomorphism then $p_K = p_F$, contradicting stability of $F$.  Therefore $f$ is an isomorphism.  

(4) Suppose $f:E\to E$ is any homomorphism.  Pick some point $x\in X$.  The linear transformation $f_x:E_x\to E_x$ has an eigenvalue $\lambda\in \C$.  Then $f-\lambda \id_E$ is not an isomorphism, so it must be zero.  Therefore $f=\lambda \id_E$.
\end{proof}

Harder-Narasimhan filtrations enable us to study arbitrary pure sheaves in terms of semistable sheaves.  Proposition \ref{prop-seesaw} is one of the important ingredients in the proof of the next theorem.

\begin{theorem-definition}[\cite{HuybrechtsLehn}]
Let $E$ be a pure sheaf of dimension $d$.  Then there is a unique filtration
$$0=E_0\subset E_1\subset \cdots \subset E_\ell = E$$ called the \emph{Harder-Narasimhan filtration} such that the quotients $\gr_i = E_i/E_{i-1}$  are semistable of dimension $d$ and reduced Hilbert polynomial $p_i$, where $$p_1 > p_2 > \cdots >p_\ell.$$
\end{theorem-definition}

In order to construct (semi)stable sheaves it is frequently necessary to also work with sheaves that are not semistable.  The next example outlines one method for constructing semistable vector bundles.  This general method was used by Dr\'ezet and Le Potier to classify the possible Hilbert polynomials of semistable sheaves on $\P^2$ \cite{LePotierLectures,DLP}.

\begin{example}\label{ex-existence}
Let $(X,H)$ be a smooth polarized projective variety.  Suppose $A$ and $B$ are vector bundles on $X$ and that the sheaf $\sHom(A,B)$ is globally generated.  For simplicity assume $r(B)-r(A) \geq \dim X$.  Let $S \subset \Hom(A,B)$ be the open subset parameterizing injective sheaf maps; this is set is nonempty since $\sHom(A,B)$ is globally generated.  Consider the family $\cE$ of sheaves on $X$ parameterized by $S$ where the sheaf $E_s$ parameterized by $s\in S$ is the cokernel $$0\to A\fto{s} B\to E_s\to 0.$$  Then for general $s\in S$, the sheaf $E_s$ is a vector bundle \cite[Proposition 2.6]{HuizengaJAG} with Hilbert polynomial $P:=P_B-P_A$.  In other words, restricting to a dense open subset $S'\subset S$, we get a family of locally free sheaves parameterized by $S'$.

Next, semistability is an open condition in families.  Thus there is a (possibly empty) open subset $S''\subset S'$ parameterizing semistable sheaves.  Let $\ell>0$ be an integer and pick polynomials $P_1,\ldots,P_\ell$ such that $P_1+\cdots +P_\ell = P$ and the corresponding reduced polynomials $p_1,\ldots,p_\ell$ have $p_1>\cdots > p_\ell$.  Then there is a locally closed subset $S_{P_1,\ldots,P_\ell} \subset S'$ parameterizing sheaves with a Harder-Narasimhan filtration of length $\ell$ with factors of Hilbert polynomial $P_1,\ldots,P_{\ell}$.  Such loci are called \emph{Shatz strata} in the base $S'$ of the family.

Finally, to show that $S''$ is nonempty, it suffices to show that the Shatz stratum $S_P$ corresponding to semistable sheaves is dense.  One approach to this problem is to show that every Shatz stratum $S_{P_1,\ldots,P_\ell}$ with $\ell \geq 2$ has codimension at least $1$.  See \cite[Chapter 16]{LePotierLectures} for an example where this is carried out in the case of $\P^2$.
\end{example}

Just as the Harder-Narasimhan filtration allows us to use semistable sheaves to build up arbitrary pure sheaves, Jordan-H\"older filtrations decompose semistable sheaves in terms of stable sheaves.  

\begin{theorem-definition}\cite{HuybrechtsLehn}
Let $E$ be a semistable sheaf of dimension $d$ and reduced Hilbert polynomial $p$.  There is a filtration $$0 = E_0 \subset E_1\subset \cdots \subset E_\ell = E$$ called the \emph{Jordan-H\"older filtration} such that the quotients $\gr_i = E_{i}/E_{i-1}$ are stable with reduced Hilbert polynomial $p$.  The filtration is not necessarily unique, but the list of stable factors is unique up to reordering.
\end{theorem-definition}

We can now precisely state the critical definition of $S$-equivalence.

\begin{definition}
Semistable sheaves $E$ and $F$ are \emph{$S$-equivalent} if they have the same list of Jordan-H\"older factors.
\end{definition}

We have already seen an example of an $S$-equivalent family of semistable sheaves in Example \ref{ex-Sequiv}, and we observed that all the parameterized sheaves must be represented by the same point in the moduli space.  In fact, the converse is also true, as the next theorem shows.

\begin{theorem}
Two semistable sheaves $E,F$ with Hilbert polynomial $P$ are represented by the same point in $M(P)$ if and only if they are $S$-equivalent.  Thus, the points of $M(P)$ are in bijective correspondence with $S$-equivalence classes of semistable sheaves with Hilbert polynomial $P$.

In particular, if there are strictly semistable sheaves of Hilbert polynomial $P$, then $M(P)$ is not a fine moduli space.
\end{theorem}

\begin{remark}
The question of when the open subset $M^s(P)$ parameterizing stable sheaves is a fine moduli space for the moduli functor $\cM^s(P)$ of stable families is somewhat delicate; in this case the points of $M^s(P)$ are in bijective correspondence with the isomorphism classes of stable sheaves, but there still need not be a universal family.

One positive result in this direction is the following.  Let ${\bf v}\in K_{\num}(X)$ be the numerical class of a stable sheaf with Hilbert polynomial $P$ (see Example \ref{ex-Chern}).  Consider the set of integers of the form $({\bf v},[F])$, where $F$ is a coherent sheaf and $(-,-)$ is the Euler pairing.  If their greatest common divisor is $1$, then $M^s({\bf v})$ carries a universal family.  (Note that the number-theoretic requirement also guarantees that there are no semistable sheaves of class ${\bf v}$.)  See \cite[\S 4.6]{HuybrechtsLehn} for details.
\end{remark}

\section{Properties of moduli spaces}\label{sec-properties} To study the birational geometry of moduli spaces of sheaves in depth it is typically necessary to have some kind of control over the geometric properties of the space.  For example, is the moduli space nonempty? Smooth? Irreducible?  What are the divisor classes on the moduli space?

Our original setup of studying a smooth projective polarized variety $(X,H)$ of any dimension is too general to get satisfactory answers to these questions.  We first mention some results on smoothness which hold with a good deal of generality, and then turn to more specific cases with far more precise results.

\subsection{Tangent spaces, smoothness, and dimension}  Let $(X,H)$ be a smooth polarized variety, and let ${\bf v}\in K_{\num}(X)$.  The tangent space to the moduli space $M=M({\bf v})$ is typically only well-behaved at points $E\in M$ parameterizing stable sheaves $E$, due to the identification of $S$-equivalence classes of sheaves in $M$.  

Let $D=\Spec \C[\varepsilon]/(\varepsilon^2)$ be the dual numbers, and let $E$ be a stable sheaf.  Then the tangent space to $M$ is the subset of $\Mor(D,M)$ corresponding to maps sending the closed point of $D$ to the point $E$.  By the moduli property, such a map corresponds to a sheaf $\cE$ on $X\times D$, flat over $D$, such that $E_0=E$.  

Deformation theory identifies the set of sheaves $\cE$ as above with the vector space $\Ext^1(E,E)$, so there is a natural isomorphism $T_EM \cong \Ext^1(E,E)$.  The obstruction to extending a first-order deformation is a class $\Ext^2(E,E)$, and if $\Ext^2(E,E)=0$ then $M$ is smooth at $E$.

For some varieties $X$ it is helpful to improve the previous statement slightly, since the vanishing $\Ext^2(E,E)=0$ can be rare, for example if $K_X$ is trivial.  If $E$ is a vector bundle, let $$\tr: \sEnd(E)\to \OO_X$$ be the \emph{trace map}, acting fiberwise as the ordinary trace of an endomorphism.  Then $$H^i(\sEnd(E))\cong \Ext^i(E,E),$$ so there are induced maps on cohomology$$\tr^i:\Ext^i(E,E)\to H^i(\OO_X).$$ We write $\Ext^i(E,E)_0 \subset \Ext^i(E,E)$ for $\ker \tr^i$, the subspace of \emph{traceless extensions}.  The subspaces $\Ext^i(E,E)_0$ can also be defined if $E$\ is just a coherent sheaf, but the construction is more delicate and we omit it.  

\begin{theorem}\label{thm-smoothness}
The tangent space to $M$ at a stable sheaf $E$ is canonically isomorphic to $\Ext^1(E,E)$, the space of first order deformations of $E$.  If $\Ext^2(E,E)_0=0$, then $M$ is smooth at $E$ of dimension $\ext^1(E,E)$.
\end{theorem}
   
We now examine several consequences of Theorem \ref{thm-smoothness} in the case of curves and surfaces.
   
\begin{example}
Suppose $X$ is a smooth curve of genus $g$, and let $M(r,d)$ be the moduli space of semistable sheaves of rank $r$ and degree $d$ on $X$.  Then the vanishing $\Ext^2(E,E)=0$ holds for any sheaf $E$, so the moduli space $M(P)$ is smooth at every point parameterizing a stable sheaf $E$.  Since stable sheaves are simple, the dimension at such a sheaf is $$\ext^1(E,E)  = 1-\chi(E,E) = r^2(g-1)+1.$$
\end{example}

\begin{example}
Let $(X,H)$ be a smooth variety, and let ${\bf v}=[\OO_X]\in K_{\num}(X)$ be the numerical class of $\OO_X$.  The moduli space $M({\bf v})$ parameterizes line bundles numerically equivalent to $\OO_X$; it is the connected component $\Pic^0 X$ of the Picard scheme $\Pic X$ which contains $\OO_X$.  For any line bundle $L\in M({\bf v})$, we have $\sEnd(L)\cong \OO_X$ and the trace map $\sEnd(L)\to \OO_X$ is an isomorphism.  Thus $\Ext^2(L,L)_0=0$, and $M({\bf v})$ is smooth of dimension $\ext^1(L,L) = h^1(\OO_X) =:q(X)$, the \emph{irregularity} of $X$.
\end{example}

\begin{example}\label{ex-dimSurface}
Suppose $(X,H)$ is a smooth surface and $E\in M^s(P)$ is a stable vector bundle.  The sheaf map $$\tr:\sEnd(E)\to \OO_X$$ is surjective, so the induced map $\tr^2:\Ext^2(E,E)\to H^2(\OO_X)$ is surjective since $X$ is a surface.  Therefore $\ext^2(E,E)_0=0$ if and only if $\ext^2(E,E) = h^2(\OO_X)$. We conclude that if $\ext^2(E,E)_0=0$ then $M(P)$ is smooth at $E$ of local dimension \begin{align*}\dim_E M(P)=\ext^1(E,E) &= 1-\chi(E,E)+\ext^2(E,E) \\&= 1-\chi(E,E)+h^2(\OO_X)
\\&= 2r^2\Delta(E)+\chi(\OO_X)(1-r^2)+q(X). \end{align*} 
\end{example}

\begin{example}
If $(X,H)$ is a smooth surface such that $H.K_X<0$, then the vanishing $\Ext^2(E,E)=0$ is automatic.  Indeed, by Serre duality,
$$\Ext^2(E,E) \cong \Hom(E,E\te K_X)^*.$$ Then $$\mu(E\te K_X) = \mu(E) + \mu(K_X) = \mu(E) + H.K_X < \mu(E),$$ so $\Hom(E,E\te K_X) = 0$ by Proposition \ref{prop-seesaw}.

The assumption $H.K_X<0$ in particular holds whenever $X$ is a del Pezzo or Hirzebruch surface.  Thus the moduli spaces $M({\bf v})$ for these surfaces are smooth at points corresponding to stable sheaves.
\end{example}

\begin{example}
If $(X,H)$ is a smooth surface and $K_X$ is trivial (e.g. $X$ is a K3 or abelian surface), then the weaker vanishing $\Ext^2(E,E)_0=0$ holds.  The trace map $\tr^2:H^2(\sEnd(E))\to H^2(\OO_X)$ is Serre dual to an isomorphism $$H^0(\OO_X)\to H^0(\sEnd(E)) = \Hom(E,E),$$ so $\tr^2$ is an isomorphism and $\Ext^2(E,E)_0=0$.  
\end{example}

\subsection{Existence and irreducibility}

What are the possible numerical invariants ${\bf v}\in K_{\num}(X)$ of a semistable sheaf on $X$?  When the moduli space is nonempty, is it irreducible? As usual, the case of curves is simplest.

\subsubsection{Existence and irreducibility for curves}
Let $M=M(r,d)$ be the moduli space of semistable sheaves of rank $r$ and degree $d$ on a smooth curve $X$ of genus $g\geq 1$.  Then $M$ is nonempty and irreducible, and unless $X$ is an elliptic curve and $r,d$ are not coprime then the stable sheaves are dense in $M$.  To show $M(r,d)$ is nonempty one can follow the basic outline of Example \ref{ex-existence}.  For more details, see \cite[Chapter 8]{LePotierLectures}. 

Irreducibility of $M(r,d)$ can be proved roughly as follows.  We may as well assume $r\geq 2$ and $d\geq 2rg$ by tensoring by a sufficiently ample line bundle.  Let $L$ denote a line bundle of degree $d$ on $X$, and consider extensions of the form $$0\to \OO_X^{r-1} \to E\to L\to 0.$$ As $L$ and the extension class vary, we obtain a family of sheaves $\cE$ parameterized by a vector bundle $S$ over the component $\Pic^d(X)$ of the Picard group.  

On the other hand, by the choice of $d$,  any semistable $E\in M(r,d)$ is generated by its global sections.  A general collection of $r-1$ sections of $E$ will be linearly independent at every $x\in X$, so that the quotient of the corresponding inclusion $\OO_X^{r-1}\to E$ is a line bundle.  Thus every semistable $E$ fits into an exact sequence as above.  The (irreducible) open subset of $S$ parameterizing semistable sheaves therefore maps onto $M(r,d)$, and the moduli space is irreducible.

\subsubsection{Existence for surfaces} For surfaces the existence question is quite subtle.  The first general result in this direction is the Bogomolov inequality.  

\begin{theorem}[Bogomolov inequality]\label{thm-bogomolov}
If $(X,H)$ is a smooth surface and $E$ is a $\mu_H$-semistable sheaf on $X$ then $$\Delta(E)\geq 0.$$
\end{theorem}  

\begin{remark}Note that the discriminant $\Delta(E)$ is independent of the particular polarization $H$, so the inequality holds for any sheaf which is slope-semistable with respect to some choice of  polarization.\end{remark}

  Recall that line bundles $L$ have $\Delta(L)=0$, so in a sense the Bogomolov inequality is sharp.  However, there are certainly Chern characters ${\bf v}$ with $\Delta({\bf v})\geq 0$ such that there is no semistable sheaf of character ${\bf v}$.  A refined Bogomolov inequality should bound $\Delta(E)$ from below in terms of the other numerical invariants of $E$.  Solutions to the existence problem for semistable sheaves on a surface can often be viewed as such improvements of the Bogomolov inequality.

\subsubsection{Existence for $\P^2$}\label{sssec-existP2}
On $\P^2$, the classification of Chern characters ${\bf v}$ such that $M({\bf v})$ is nonempty has been carried out by Dr\'ezet and Le Potier \cite{DLP,LePotierLectures}.  A \emph{(semi)exceptional bundle} is a rigid (semi)stable bundle, i.e. a (semi)stable bundle with $\Ext^1(E,E)=0$.  Examples of exceptional bundles include line bundles, the tangent bundle $T_{\P^2}$, and infinitely more examples obtained by a process of \emph{mutation}.  The dimension formula for a moduli space of sheaves on $\P^2$ reads $$\dim M({\bf v}) = r^2(2\Delta-1)+1,$$ so an exceptional bundle has discriminant $\Delta = \frac{1}{2}-\frac{1}{2r^2}<\frac{1}{2}.$  The dimension formula suggests an immediate refinement of the Bogomolov inequality: if $E$ is a non-exceptional stable bundle, then $\Delta(E)\geq \frac{1}{2}$.    

However, exceptional bundles can provide even stronger Bogomolov inequalities for non-exceptional bundles.  For example, suppose $E$ is a semistable sheaf with $0<\mu(E) <1$.  Then $\Hom(E,\OO_X)=0$ and $$\Ext^2(E,\OO_X)\cong \Hom(\OO_X,E\te K_X)^*=0$$ by semistability and Proposition \ref{prop-seesaw}.  Thus $\chi(E,\OO_X)\leq 0$.  By the Riemann-Roch theorem, this inequality is equivalent to the inequality $$\Delta(E) \geq P(-\mu(E))$$ where $P(x) = \frac{1}{2}x^2+\frac{3}{2}x+1$; this inequality is stronger than the ordinary Bogomolov inequality for any $\mu(E)\in (0,1)$.  
 
 Taking all the various exceptional bundles on $\P^2$ into account in a similar manner, one defines a function $\delta:\R\to \R$ with the property that any non-semiexceptional semistable bundle $E$ satisfies $\Delta(E)\geq \delta(\mu(E))$.  The graph of $\delta$ is Figure \ref{fig-deltaCurve}.  Dr\'ezet and Le Potier prove the converse theorem:  exceptional bundles are the only obstruction to the existence of stable bundles with given numerical invariants.
\begin{theorem}
Let ${\bf v}$ be an integral Chern character on $\P^2$.  There is a non-exceptional stable vector bundle on $\P^2$ with Chern character ${\bf v}$ if and only if $\Delta({\bf v})\geq \delta(\mu({\bf v}))$.
\end{theorem}
The method of proof follows the outline indicated in Example \ref{ex-existence}. \begin{figure}[t]
\begin{center}
\includegraphics[scale=0.95]{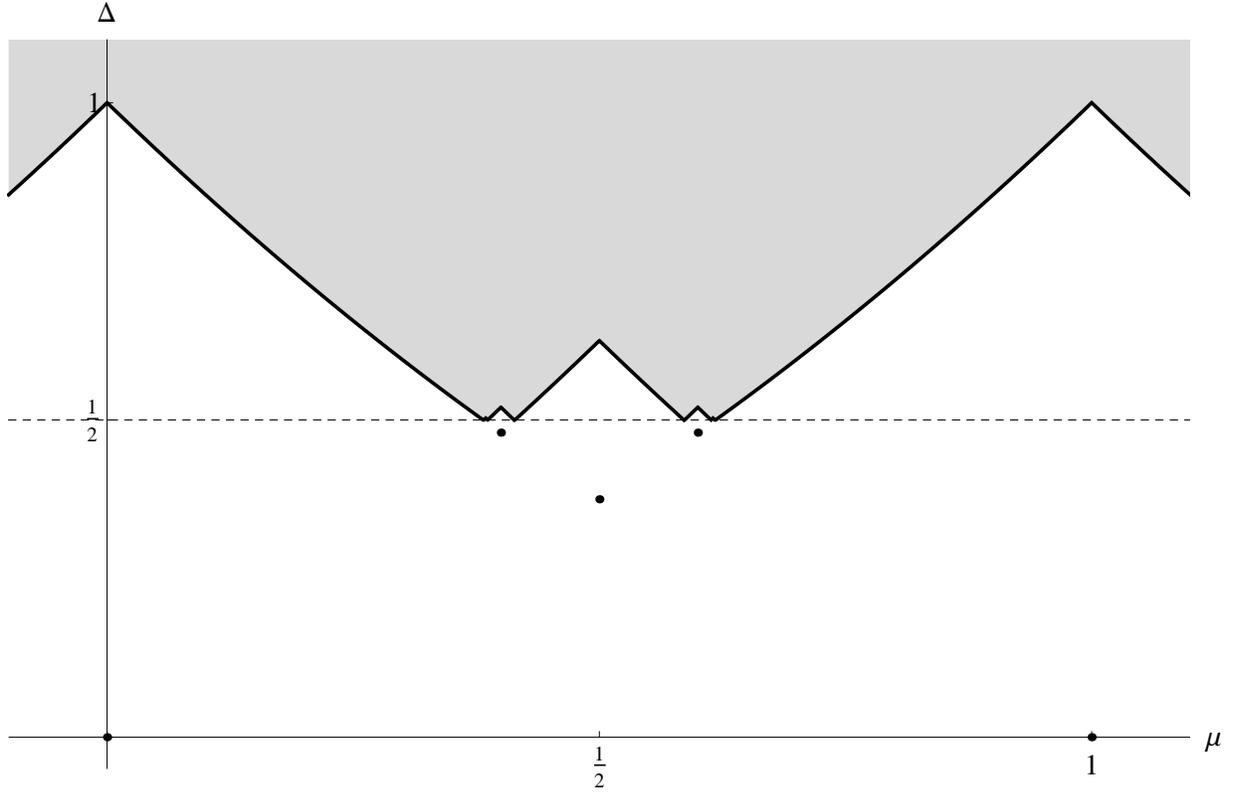}
\end{center}
\caption{The curve $\delta(\mu)$ occurring in the classification of stable bundles on $\P^2$.  If $(r,\mu,\Delta)$ are the invariants of an integral Chern character, then there is a non-exceptional stable bundle $E$ with these invariants if and only if $\Delta\geq \delta(\mu)$.  The invariants of the first several exceptional bundles are also displayed.}
\label{fig-deltaCurve}
\end{figure}

\subsubsection{Existence for other rational surfaces} In the case of $X=\P^1\times \P^1$, Rudakov \cite{Rudakov1,Rudakov2} gives a solution to the existence problem that is similar to the Dr\'ezet-Le Potier result for $\P^2$.  However, the geometry of exceptional bundles is more complicated than for $\P^2$, and as a result the classification is somewhat less explicity.
To our knowledge a satisfactory answer to the existence problem has not yet been given for a Del Pezzo or Hirzebruch surface.

\subsubsection{Irreducibility for rational surfaces}
For many rational surfaces $X$ it is known that the moduli space $M_H({\bf v})$ is irreducible.  One common argument is to introduce a mild relaxation of the notion of semistability and show that the stack parameterizing such objects is irreducible and contains the semistable sheaves as an open dense substack.

For example, Hirschowitz and Laszlo \cite{HirschowitzLaszlo} introduce the notion of a \emph{prioritary sheaf} on $\P^2$.  A torsion-free coherent sheaf $E$ on $\P^2$ is prioritary if $$\Ext^2(E,E(-1)) = 0.$$  By Serre duality, any torsion-free sheaf whose Harder-Narasimhan factors have slopes that are ``not too far apart'' will be prioritary, so it is very easy to construct prioritary sheaves. For example, semistable sheaves are prioritary, and sheaves of the form $\OO_{\P^2}(a)^{\oplus k} \oplus \OO_{\P^2}(a+1)^{\oplus l}$ are prioritary.  The class of prioritary sheaves is also closed under elementary modifications, which makes it possible to study them by induction on the Euler characteristic as in Example \ref{ex-elementary}.

The Artin stack $\cP({\bf v})$ of prioritary sheaves with invariants ${\bf v}$ is smooth, essentially because $\Ext^2(E,E)=0$ for any prioritary sheaf.  There is a unique prioritary sheaf of a given slope and rank with minimal discriminant, given by a sheaf of the form $\OO_{\P^2}(a)^{\oplus k} \oplus \OO_{\P^2}(a+1)^{\oplus l}$ with the integers $a,k,l$ chosen appropriately.  Hirschowitz and Laszlo show that any connected component of $\cP({\bf v})$ contains a sheaf which is an elementary modification of another sheaf.  By induction on the Euler characteristic, they conclude that $\cP({\bf v})$ is connected, and therefore irreducible.  Since semistability is an open property, the stack $\cM({\bf v})$ of semistable sheaves is an open substack of $\cP({\bf v})$ and therefore dense and irreducible if it is nonempty.  Thus the coarse space $M({\bf v})$ is irreducible as well.

Walter \cite{Walter} gives another argument establishing the irreducibility of the moduli spaces $M_H({\bf v})$ on a Hirzebruch surface whenever they  are nonempty.  The arguments make heavy use of the ruling, and study the stack of sheaves which are prioritary with respect to the fiber class.  In more generality, he also studies the question of irreducibility on a geometrically ruled surface, at least under a condition on the polarization which ensures that semistable sheaves are prioritary with respect to the fiber class.

\subsubsection{Existence and irreducibility for K3's}\label{sssec-K3exist}
By work of Yoshioka, Mukai, and others, the existence problem has a particularly simple and beautiful solution when $(X,H)$ is a smooth K3 surface (see \cite{Yoshioka3}, or \cite{BM2,BM3} for a simple treatment).  Define the \emph{Mukai pairing} $\langle-,-\rangle$ on $K_{\num}(X)$ by $\langle {\bf v},{\bf w}\rangle = -\chi({\bf v},{\bf w})$; we can make sense of this formula by the same method as in Example \ref{ex-Chern}.  Since $X$ is a K3 surface, $K_X$ is trivial and the Mukai pairing is symmetric by Serre duality.  By Example \ref{ex-dimSurface}, if there is a stable sheaf $E$ with invariants ${\bf v}$ then the moduli space $M({\bf v})$ has dimension $2+\langle {\bf v},{\bf v}\rangle$ at $E$.  If $E$ is a stable sheaf of class ${\bf v}$ with $\langle {\bf v},{\bf v}\rangle = -2$, then $E$ is called \emph{spherical} and the moduli space $M_H({\bf v})$ is a single reduced point.

A class ${\bf v}\in K_{\num}(X)$ is called \emph{primitive} if it is not a multiple of another class.  If the polarization $H$ of $X$ is chosen suitably generically, then ${\bf v}$ being primitive ensures that there are no strictly semistable sheaves of class ${\bf v}$.  Thus, for a generic polarization, a necessary condition for the existence of a stable sheaf is that $\langle {\bf v},{\bf v}\rangle \geq -2$.

\begin{definition}
A primitive class ${\bf v}=(r,c,d)\in K_{\num}(X)$ is called \emph{positive} if $\langle {\bf v},{\bf v}\rangle \geq -2$ and either 
\begin{enumerate}
\item\label{req1} $r >0$, or
\item $r = 0$ and $c$ is effective, or
\item\label{req3} $r = 0$, $c=0$, and $d>0$.
\end{enumerate}
\end{definition}

The additional requirements (\ref{req1})-(\ref{req3}) in the definition are automatically satisfied any time there is a \emph{sheaf} of class ${\bf v}$, so they are very mild.  

\begin{theorem}\label{thm-k3exist}
Let $(X,H)$ be a smooth K3 surface.  Let ${\bf v}\in K_{\num}(X)$, and write ${\bf v} = m {\bf v}_0$, where ${\bf v}_0$ is primitive and $m$ is a positive integer.  

If ${\bf v}_0$ is positive, then the moduli space $M_H({\bf v})$ is nonempty.  If furthermore $m=1$ and the polarization $H$ is sufficiently generic, then $M_H({\bf v})$ is a smooth, irreducible, holomorphic symplectic variety.

If $M_H({\bf v})$ is nonempty and the polarization is sufficiently generic, then ${\bf v}_0$ is positive.  
\end{theorem}

The Mukai pairing can be made particularly simple from a computational standpoint by studying it in terms of a different coordinate system.  Let $$H^*_\alg(X) = H^0(X,\Z)\oplus \NS(X) \oplus H^4(X,\Z).$$ Then there is an isomorphism $v:K_{\num}(X)\to H^*_{\alg}(X,\Z)$ defined by $v({\bf v}) = {\bf v}\cdot \sqrt{\td(X)}.$ The vector $v({\bf v})$ is called a \emph{Mukai vector}.  The Todd class $\td(X)\in H_\alg^*(X)$ is  $(1,0,2)$, so $\sqrt{\td(X)} = (1,0,1)$ and $$v({\bf v}) = (\ch_0({\bf v}),\ch_1({\bf v}),\ch_0({\bf v})+\ch_2({\bf v}))=(r,c_1,r+\frac{c_1^2}{2}-c_2).$$ Suppose ${\bf v},{\bf w}\in K_{\num}(X)$ have Mukai vectors $v({\bf v}) = (r,c,s)$, $v({\bf w}) = (r',c',s')$.  Since $\sqrt{\td(X)}$ is self-dual, the Hirzebruch-Riemann-Roch theorem gives $$\langle {\bf v},{\bf w}\rangle = -\chi({\bf v},{\bf w}) = -\int_X {\bf v}^*\cdot {\bf w} \cdot \td(X) = -\int_X (r,-c,s)\cdot (r',c',s') = cc'-rs'-r's.$$

It is worth pointing out that Theorem \ref{thm-k3exist} can also be stated as a strong Bogomolov inequality, as in the Dr\'ezet-Le Potier result for $\P^2$.  Let ${\bf v}_0$ be a primitive vector which is the vector of a coherent sheaf.  The irregularity of $X$ is $q(X)=0$ and $\chi(\OO_X)=2$, so as in Example \ref{ex-dimSurface} $$\langle {\bf v}_0,{\bf v}_0 \rangle  = 2r^2\Delta({\bf v}_0) +2(1-r^2)-2 = 2r^2(\Delta({\bf v}_0)-1).$$ Therefore, ${\bf v}_0$ is positive and non-spherical if and only if $\Delta({\bf v}_0)\geq 1$.

\subsubsection{General surfaces}\label{ssec-Hilb}
On an arbitrary smooth surface $(X,H)$ the basic geometry of the moduli space is less understood.  To obtain good results, it is necessary to impose some kind of additional hypotheses on the Chern character ${\bf v}$.

For one possibility, we can take ${\bf v}$ to be the character of an ideal sheaf $I_Z$ of a zero-dimensional scheme $Z\subset X$ of length $n$.  Then the moduli space of sheaves of class ${\bf v}$ with determinant $\OO_X$ is the \emph{Hilbert scheme} of $n$ points on $X$, written $X^{[n]}$.  It parameterizes ideal sheaves of subschemes $Z\subset X$ of length $n$.  

\begin{remark}
Note that any rank $1$ torsion-free sheaf $E$ with determinant $\OO_X$ admits an inclusion $E\to E^{**} := \det E  =\OO_X$, so that $E$ is actually an ideal sheaf.  Unless $X$ has irregularity $q(X) = 0$, the Hilbert scheme $X^{[n]}$ and moduli space $M({\bf v})$ will differ, since the latter space also contains sheaves of the form $L\te I_Z$, where $L$ is a line bundle numerically equivalent to $\OO_X$.  In fact, $M({\bf v}) \cong X^{[n]}\times \Pic^0(X)$.
\end{remark}

Classical results of Fogarty show that Hilbert schemes of points on a surface are very well-behaved.

\begin{theorem}[\cite{Fogarty1}]
The Hilbert scheme of points $X^{[n]}$ on a smooth surface $X$ is smooth and irreducible.  It is a fine moduli space, and carries a universal ideal sheaf.
\end{theorem}

At the other extreme, if the rank is arbitrary then there are \emph{O'Grady-type} results which show that the moduli space has many good properties if we require the discriminant of our sheaves to be sufficiently large. 

\begin{theorem}[\cite{HuybrechtsLehn,OGrady}]
There is a constant $C$ depending on $X,H$, and $r$, such that if ${\bf v}$ has rank $r$ and $\Delta({\bf v}) \geq C$ then the moduli space $M_H({\bf v})$ is nonempty, irreducible, and normal.  The $\mu$-stable sheaves $E$ such that $\ext^2(E,E)_0=0$ are dense in $M_H({\bf v})$, so $M_H({\bf v})$ has the expected dimension $$\dim M_H({\bf v}) = 2r^2\Delta(E)+\chi(\OO_X)(1-r^2)+q(X).$$
\end{theorem}

\section{Divisors and classical birational geometry}\label{sec-classical}

In this section we introduce some of the primary objects of study in the birational geometry of varieties.  We then study some simple examples of the birational geometry of moduli spaces from the classical point of view.

\subsection{Cones of divisors}  Let $X$ be a normal projective variety.  Recall that $X$ is \emph{factorial} if every Weil divisor on $X$ is Cartier, and $\Q$-factorial if every Weil divisor has a multiple that is Cartier.  To make the discussion in this section easier we will assume that $X$ is $\Q$-factorial.  This means that describing a codimension $1$ locus on $X$ determines the class of a $\Q$-Cartier divisor.

\begin{definition}
Two Cartier divisors $D_1,D_2$ (or $\Q$- or $\R$-Cartier divisors) are \emph{numerically equivalent}, written $D_1 \equiv D_2$,  if $D_1\cdot C = D_2 \cdot C$ for every curve $C\subset X$.  The \emph{Neron-Severi space} $N^1(X)$ is the real vector space $\Pic(X)\te \R/\equiv$.
\end{definition}

\subsubsection{Ample and nef cones}The first object of study in birational geometry is the \emph{ample cone} $\Amp(X)$ of $X$.  Roughly speaking, the ample cone parameterizes the various projective embeddings of $X$.  A Cartier divisor $D$ on $X$ is \emph{ample} if the map to projective space determined by $\OO_X(mD)$ is an embedding for sufficiently large $m$.  The Nakai-Moishezon criterion for ampleness says that $D$ is ample if and only if $D^{\dim V}.V >0$ for every subvariety $V\subset X$.  In particular, ampleness only depends on the numerical equivalence class of $D$.  A positive linear combination of ample divisors is also ample, so it is natural to consider the cone spanned by ample classes.

\begin{definition}
The \emph{ample cone} $\Amp(X)\subset N^1(X)$ is the open convex cone spanned by the numerical classes of ample Cartier divisors.

An $\R$-Cartier divisor $D$ is ample if its numerical class is in the ample cone.
\end{definition}

From a practical standpoint it is often easier to work with \emph{nef} (i.e. \emph{numerically effective}) divisors instead of ample divisors.  We say that a Cartier divisor $D$ is nef if $D.C\geq 0$ for every curve $C\subset X$.  This is clearly a numerical condition, so nefness extends easily to $\R$-divisors and they span a cone $\Nef(X)$, the \emph{nef cone} of $X$.  By Kleiman's theorem, the problems of studying ample or nef cones are essentially equivalent.

\begin{theorem}[{\cite[Theorem 1.27]{Debarre}}]
The nef cone is the closure of the ample cone, and the ample cone is the interior of the nef cone:
$$ \Nef(X) = \overline{\Amp(X)} \qquad \mbox{and} \qquad \Amp(X) = \Nef(X)^{\circ} .$$
\end{theorem}

Nef divisors are particularly important in birational geometry because they record the behavior of the simplest nontrivial morphisms to other projective varieties, as the next example shows.

\begin{example}\label{ex-nefpullback}
Suppose $f:X\to Y$ is any morphism of projective varieties.  Let $L$ be a very ample line bundle on $Y$, and consider the line bundle $f^*L$.  If $C\subset X$ is any irreducible curve, we can find an effective divisor $D\subset Y$ representing $L$ such that the image of $C$ is not contained entirely in $D$.  This implies $C.(f^*L) \geq 0$, so $f^*L$ is nef.  Note that if $f$ contracts some curve $C\subset X$ to a point, then $C.(f^*L)=0$, so $f^*L$ is on the boundary of the nef cone.

As a partial converse, suppose $D$ is a nef divisor on $X$ such that the linear series $|mD|$ is base point free for some $m>0$; such a divisor class is called \emph{semiample}.  Then for sufficiently large and divisible $m$, the image of the map $\phi_{|mD|}:X\to |mD|^*$ is a projective variety $Y_m$ carrying an ample line bundle $L$ such that $\phi_{|mD|}^* L = \OO_X(mD).$  See \cite[Theorem 2.1.27]{Lazarsfeld} for details and a more precise statement.
\end{example}

\begin{example}\label{ex-nefCompute}
Classically, to compute the nef (and hence ample) cone of a variety $X$ one typically first constructs a subcone $\Lambda \subset \Nef(X)$ by finding divisors $D$ on the boundary arising from interesting contractions $X\to Y$ as in Example \ref{ex-nefpullback}.  One then dually constructs interesting curves $C$ on $X$ to span a cone $\Nef(X)\subset \Lambda'$ given as the divisors intersecting the curves nonnegatively.  If enough divisors and curves are constructed so that $\Lambda = \Lambda'$, then they equal the nef cone.

One of the main features of the positivity lemma of Bayer and Macr\`i will be that it produces nef divisors on moduli spaces of sheaves $M$ without having to worry about finding a map $M\to Y$ to a projective variety giving rise to the divisor.  A priori these nef divisors may not be semiample or have sections at all, so it may or may not be possible to construct these divisors and prove their nefness via more classical constructions.  See \S\ref{sec-positivity} for more details.
\end{example}

\begin{example}
For an easy example of the procedure in Example \ref{ex-nefCompute}, consider the blowup $X = \Bl_p \P^2$ of $\P^2$ at a point $p$.  Then $\Pic X \cong \Z H\oplus \Z E$, where $H$ is the pullback of a line under the map $\pi:X\to \P^2$ and $E$ is the exceptional divisor.  The Neron-Severi space $N^1(X)$ is the two-dimensional real vector space spanned by $H$ and $E$.  Convex cones in $N^1(X)$ are spanned by two extremal classes.

Since $\pi$ contracts $E$, the class $H$ is an extremal nef divisor.
We also have a fibration $f: X\to \P^1$, where the fibers are the proper transforms of lines through $p$.  The pullback of a point in $\P^1$ is of class $H-E$, so $H-E$ is an extremal nef divisor.  Therefore $\Nef(X)$ is spanned by $H$ and $H-E$.
\end{example}

\subsubsection{(Pseudo)effective and big cones}  The easiest interesting space of divisors to define is perhaps the \emph{effective cone} $\Eff(X)  \subset N^1(X)$, defined as the subspace spanned by numerical classes of effective divisors.  Unlike nefness and ampleness, however, effectiveness is \emph{not} a numerical property: for instance, on an elliptic curve $C$, a line bundle of degree $0$ has an effective multiple if and only if it is torsion. 

The effective cone is in general neither open nor closed.  Its closure $\overline\Eff(X)$ is less subtle, and called the \emph{pseudo-effective cone}.  The interior of the effective cone is the \emph{big cone} $\Bigc(X)$, spanned by divisors $D$ such that the linear series $|mD|$ defines a map $\phi_{|mD|}$ whose image has the same dimension as $X$.  Thus, big divisors are the natural analog of birational maps.  By Kodaira's Lemma \cite[Proposition 2.2.6]{Lazarsfeld}, bigness is a numerical property.

\begin{example}
The strategy for computing pseudoeffective cones is typically similar to that for computing nef cones.  On the one hand, one constructs effective divisors to span a cone $\Lambda\subset \overline\Eff(X)$.  A \emph{moving curve} is a numerical curve class $[C]$ such that irreducible representatives of the class pass through a general point of $X$.  Thus if $D$ is an effective divisor we must have $D.C\geq 0$; otherwise $D$ would have to contain every irreducible curve of class $C$.  Thus the moving curve classes dually determine a cone $\overline{\Eff}(X) \subset \Lambda'$, and if $\Lambda=\Lambda'$ then they equal the pseudoeffective cone.  This approach is justified by the seminal work of Boucksom-Demailly-P\u aun-Peternell, which establishes a duality between the pseudoeffective cone and the cone of moving curves \cite{BDPP}.
\end{example}

\begin{example}
On $X = \Bl_p\P^2$, the curve class $H$ is moving and $H.E = 0$.  Thus $E$ spans an extremal edge of $\Eff(X)$.  The curve class $H-E$ is also moving, and $(H-E)^2 = 0$.  Therefore $H-E$ spans the other edge of $\Eff(X)$, and $\Eff(X)$ is spanned by $H-E$ and $E$.
\end{example}

\subsubsection{Stable base locus decomposition}

The nef cone $\Nef(X)$ is one chamber in a decomposition of the entire pseudoeffective cone $\overline\Eff(X)$.  By the \emph{base locus} $\Bs(D)$ of a divisor $D$ we mean the base locus of the complete linear series $|D|$, regarded as a subset (i.e. not as a subscheme) of $X$.  By convention, $\Bs(D) = X$ if $|D|$ is empty.  The \emph{stable base locus} of $D$ is the subset $$\BBs(D) = \bigcap_{m >0} \Bs(D)$$ of $X$.  One can show that $\BBs(D)$ coincides with the base locus $\Bs (mD)$ of sufficiently large and divisible multiples $mD.$

\begin{example}
The base locus and stable base locus of $D$ depend on the class of $D$ in $\Pic(X)$, not just on the numerical class of $D$.  For example, if $L$ is a degree $0$ line bundle on an elliptic curve $X$, then $\Bs(L) = X$ unless $L$ is trivial, and $\BBs(L)=X$ unless $L$ is torsion in $\Pic(X)$.
\end{example}

Since (stable) base loci do not behave well with respect to numerical equivalence, for the rest of this subsection we assume $q(X) = 0$ so that linear and numerical equivalence coincide and $N^1(X)_\Q = \Pic(X)\te \Q$.  Then the pseudoeffective cone $\overline{\Eff}(X)$ has a wall-and-chamber decomposition where the stable base locus remains constant on the open  chambers.  These various chambers control the birational maps from $X$ to other projective varieties.  For example, if $f:X\dashrightarrow Y$ is the rational map given by a sufficiently divisible multiple $|mD|$, then the indeterminacy locus of the map is contained in the stable base locus.

\begin{example}
Stable base loci decompositions are typically computed as follows.  First, one constructs effective divisors in a multiple $|mD|$ and takes their intersection to get a variety $Y$ with $\BBs(D) \subset Y$.  In the other direction, one looks for curves $C$ on $X$ such that $C.D <0$.  Then any divisor of class $mD$ must contain $C$, so $\BBs(D)$ contains every curve numerically equivalent to $C$.
\end{example}

When the Picard rank of $X$ is two, the chamber decompositions can often be made very explicit.  In this case it is notationally conventient to write, for example, $(D_1,D_2]$ to denote the cone of divisors of the form $a_1D_1+a_2D_2$ with $a_1>0$ and $a_2\geq 0$.

\begin{example}
Let $X = \Bl_p \P^2$.  The nef cone is $[H,H-E]$, and both $H,H-E$ are basepoint free.  Thus the stable base locus is empty in the closed chamber $[H,H-E]$. If $D\in (H,E]$ is an effective divisor, then $D.E <0$, so $D$ contains $E$ as a component.  The stable base locus of divisors in the chamber $(H,E]$ is $E$.
\end{example}

We now begin to investigate the birational geometry of some of the simplest moduli spaces of sheaves on surfaces from a classical point of view.

\subsection{Birational geometry of Hilbert schemes of points}  Let $X$ be a smooth surface with irregularity $q(X)=0$, and let ${\bf v}$ be the Chern character of an ideal sheaf $I_Z$ of a collection $Z$ of $n$ points.  Then $M({\bf v})$ is the Hilbert scheme $X^{[n]}$ of $n$ points on $X$, parameterizing zero-dimensional schemes of length $n$.  See \S\ref{ssec-Hilb} for its basic properties.

\subsubsection{Divisor classes} Divisor classes on the Hilbert scheme $X^{[n]}$ can be understood entirely in terms of the birational \emph{Hilbert-Chow morphism} $h:X^{[n]}\to X^{(n)}$ to the symmetric product $X^{(n)} = \Sym^n X$.  Informally, this map sends the ideal sheaf of $Z$ to the sum of the points in $Z$, with multiplicities given by the length of the scheme at each point.

\begin{remark}
The symmetric product $X^{(n)}$ can itself be viewed as the moduli space of $0$-dimensional sheaves with Hilbert polynomial $P(m)=n$.  Suppose $E$ is a zero-dimensional sheaf with constant Hilbert polynomial $\ell$ and that $E$ is supported at a single point $p$.  Then $E$ admits a length $\ell$\ filtration where all the quotients are isomorphic to $\OO_p$.  Thus, $E$ is $S$-equivalent to $\OO_p^{\oplus \ell}$.  Since $S$-equivalent sheaves are identified in the moduli space, the moduli space $M(P)$ is just $X^{(n)}$.

The Hilbert-Chow morphism $h:X^{[n]}\to X^{(n)}$ can now be seen to come from the moduli property for $X^{(n)}$.  Let $\cI$ be the universal ideal sheaf on $X\times X^{[n]}$.  The quotient of the inclusion $\cI\to \OO_{X\times X^{[n]}}$ is then a family of zero-dimensional sheaves of length $n$.  This family induces a map $X^{[n]}\to X^{(n)}$, which is just the Hilbert-Chow morphism.
\end{remark}

The exceptional locus of the Hilbert-Chow morphism is a divisor class $B$ on the Hilbert scheme $X^{[n]}$.  Alternately, $B$ is the locus of nonreduced schemes.  It is swept out by curves contained in fibers of the Hilbert-Chow morphism.  A simple example of such a curve is given by fixing $n-2$ points in $X$ and allowing a length $2$ scheme $\Spec \C[\varepsilon]/(\varepsilon^2)$ to ``spin'' at one additional point.

\begin{remark}
The divisor class $B/2$ is also Cartier, although it is not effective so it is harder to visualize.  Let $\cZ\subset X\times X^{[n]}$ denote the universal subscheme of length $n$, and let $p:\cZ\to X$ and $q:\cZ\to X^{[n]}$ be the projections.  Then the \emph{tautological bundle} $q_*p^*\OO_X$ is a rank $n$ vector bundle with determinant of class $-B/2$.
\end{remark}

Any line bundle $L$ on $X$ induces a line bundle $L^{(n)}$ on the symmetric product.  Pulling back this line bundle by the Hilbert-Chow morphism gives a line bundle $L^{[n]}:= h^*L^{(n)}$.  This gives an inclusion $\Pic(X)\to \Pic(X^{[n]})$.  If $L$ can be represented by a reduced effective divisor $D$, then $L^{[n]}$ can be represented by the locus $$D^{[n]} :=  \{Z\in X^{[n]} :Z\cap D \neq \emptyset\}.$$ Fogarty proves that the divisors mentioned so far generate the Picard group.

\begin{theorem}[Fogarty \cite{Fogarty2}] Let $X$ be a smooth surface with  $q(X) = 0$.  Then $$\Pic(X^{[n]}) \cong \Pic(X) \oplus \Z(B/2).$$ Thus, tensoring by $\R$, $$N^1(X^{[n]}) \cong N^1(X) \oplus \R B.$$
\end{theorem}

There is another interesting way to use a line bundle on $X$ to construct effective divisor classes.  In examples, many extremal effective divisors can be realized in this way.

\begin{example}
Suppose $L$ is a line bundle on $X$ with $m:=h^0(L)> n$.  If $Z\subset X$ is a general subscheme of length $n$, then $H^0(L\te I_Z)\subset H^0(L)$ is a subspace of codimension $n$.  Thus we get a rational map $$\phi: X^{[n]}\dashrightarrow G := \Gr(m-n,m)$$ to the Grassmannian $G$ of codimension $n$ subspaces of $H^0(L)$.  The line bundle $\tilde L^{[n]} := \phi^*\OO_G(1)$ (which is well-defined since the indeterminacy locus of $\phi$ has codimension at least $2$) can be represented by an effective divisor as follows.  Let $W\subset H^0(L)$ be a sufficiently general subspace of dimension $n$; one frequently takes $W$ to be the subspace of sections of $L$ passing through $m-n$ general points.  Then the locus $$\tilde D^{[n]} = \{Z\in X^{[n]} : H^0(L\te I_Z) \cap W \neq \{0\}\}$$ is an effective divisor representing $\phi^*\OO_G(1)$.
\end{example}

\subsubsection{Curve classes}  Let $C\subset X$ be an irreducible curve.  There are two immediate ways that we can induce a curve class on $X^{[n]}$.

\begin{example}
Fix $n-1$ points $p_1,\ldots,p_{n-1}$ on $X$ which are not in $C$.  Allowing an $n$th point $p_n$ to travel along $C$ gives a curve $\tilde C_{[n]} \subset X^{[n]}$.
\end{example}

\begin{example} Suppose $C$ admits a $g_n^1$.  If the $g_n^1$ is base-point free, then we get a degree $n$ map $C\to \P^1$.  The fibers of this map induce a rational curve $\P^1\to X^{[n]}$, and we write $C_{[n]}$ for the class of the image.  If the $g_n^1$ is not base-point free, we can first remove the basepoints to get a map $\P^1\to X^{[m]}$ for some $m<n$, and then glue the basepoints back on to get a map $\P^1\to X^{[n]}$.  The class $C_{[n]}$ doesn't depend on the particular $g_n^1$ used to construct the curve (see for example \cite[Proposition 3.5]{HuizengaThesis} in the case of $\P^2$).
\end{example}

\begin{remark}
Typically the curve classes $C_{[n]}$ are more interesting than $\tilde C_{[n]}$ and they frequently show up as extremal curves in the cone of curves.  However, the class $C_{[n]}$ is only defined if $C_{[n]}$ carries an interesting linear series of degree $n$, while $\tilde C_{[n]}$ always makes sense; thus curves of class $\tilde C_{[n]}$ are also sometimes used.
\end{remark}

Both curve classes $\tilde C_{[n]}$ and $C_{[n]}$ have the useful property that the intersection pairing with divisors is preserved, in the sense that if $D\subset X$ is a divisor then $$D^{[n]}.\tilde C_{[n]} = D^{[n]}.C_{[n]} = D.C;$$ indeed, it suffices to check the equalities when $D$ and $C$ intersect transversely, and in that case $D^{[n]}$ and $C_{[n]}$ (resp. $\tilde C_{[n]}$) intersect transversely in $D.C$ points.  

The intersection with $B$ is more interesting.  Clearly $$\tilde C_{[n]}.B = 0.$$ On the other hand, the nonreduced schemes parameterized by a curve of class $C_{[n]}$ correspond to ramification points of the degree $n$ map $C\to \P^1$.  The Riemann-Hurwitz formula then implies $$C_{[n]}.B = 2g(C)-2+2n.$$
One further curve class is useful; we write $C_0$ for the class of a curve contracted by the Hilbert-Chow morphism.  

\subsubsection{The intersection pairing} At this point we have collected enough curve and divisor classes to fully determine the intersection pairing between curves and divisors and find relations between the various classes.
The classes $C_0$ and $C_{[n]}$ for $C$ any irreducible curve span $N_1(X)$, so to completely compute the intersection pairing we are only missing the intersection number $C_0.B$.  However, since this intersection number is negative, we use the additional curve and divisor classes $\tilde C_{[n]}$ and $\tilde D^{[n]}$ to compute this number.  To this end, we compute the intersection numbers of $\tilde D^{[n]}$ with our curve classes.  

\begin{example}\label{ex-divCurveInt1}
To compute $\tilde D^{[n]}.C_0$, let $m = h^0(\OO_X(D))$, fix $m-n$ general points $p_1,\ldots,p_{m-n}$ in $X$, and represent $\tilde D^{[n]}$ as the set of schemes $Z$ such that there is a curve on $X$ of class $D$ passing through $p_1,\ldots,p_{m-n}$ and $Z$.  Schemes parameterized by $C_0$ are supported at $n-1$ general points $q_1,\ldots,q_{n-1}$, with a spinning tangent vector at $q_{n-1}$.   There is a unique curve $D'$ of class $D$ passing through $p_1,\ldots,p_{m-n},q_1,\ldots,q_{n-1}$, and it is smooth at $q_{n-1}$, so there is a single point of intersection between $C_0$ and $\tilde D^{[n]}$, occurring when the tangent vector at $q_{n-1}$ is tangent to $D'$.  Thus $\tilde D^{[n]}.C_0 = 1$.
\end{example}

\begin{example}
Next we compute $\tilde C_{[n]}.\tilde D^{[n]}$.  Represent $\tilde D^{[n]}$ as in Example \ref{ex-divCurveInt1}.  The curve class $\tilde C_{[n]}$ is represented by fixing $n-1$ points $q_1,\ldots,q_{n-1}$ and letting $q_n$ travel along $C$.  There is a unique curve $D'$ of class $D$ passing through $p_1,\ldots,p_{m-n},q_1,\ldots,q_{n-1}$, so $\tilde C_{[n]}$ meets $\tilde D^{[n]}$ when $q_n \in C\cap D'$.  Thus $\tilde C_{[n]}.\tilde D^{[n]} = C.D$.
\end{example}

\begin{example}
For an irreducible curve $C\subset X$, write $\hat C_{[n]}$ for the curve class on $X^{[n]}$ obtained by fixing $n-2$ general points in $X$, fixing one point on $C$, and letting one point travel along $C$ and collide with the point fixed on $C$.  It follows immediately that
\begin{align*}
\hat C_{[n]}.D^{[n]} &= C.D\\
\hat C_{[n]}.\tilde D^{[n]} &= C.D -1.
\end{align*}
Less immediately, we find $\hat C_{[n]}.B = 2$: while the curve meets $B$ set-theoretically in one point, a tangent space calculation shows this intersection has multiplicity $2$.
\end{example}
We now collect our known intersection numbers.
$$\begin{array}{c|ccc}
&D^{[n]} & \tilde D^{[n]} & B\\
\hline C_{[n]} & C.D && 2g(C)-2+2n\\
 \tilde C_{[n]}  & C.D & C.D& 0 \\
 \hat C_{[n]} & C.D & C.D-1 & 2\\
 C_0 & 0 & 1
\end{array}$$

As $\tilde D^{[n]}. C_0 \neq 0$, the divisors $\tilde D^{[n]}$ are all not in the codimension one subspace $N^1(X)\subset N^1(X^{[n]})$.  Therefore  the divisor classes of type $D^{[n]}$ and $\tilde D^{[n]}$ together span $N^1(X)$. It now follows that $$C_0+\hat C_{[n]} = \tilde C_{[n]}$$ since both sides pair the same with divisors $D^{[n]}$ and $\tilde D^{[n]}$, and thus $C_0.B=-2$.  We then also find relations $$C_{[n]} = \tilde C_{[n]}-(g(C)-1+n)C_0$$ and $$\tilde D^{[n]} = D^{[n]} - \frac{1}{2}B.$$ In particular, the divisors of type $\tilde D^{[n]}$ are all in the half-space of divisors with negative coefficient of $B$ in terms of the Fogarty isomorphism $N^1(X^{[n]}) \cong N^1(X) \oplus \R B$.   We can also complete our intersection table.
$$\begin{array}{c|ccc}
&D^{[n]} & \tilde D^{[n]} & B\\
\hline C_{[n]} & C.D &C.D -(g(C)-1+n)& 2g(C)-2+2n\\
 \tilde C_{[n]}  & C.D & C.D& 0 \\
 \hat C_{[n]} & C.D & C.D-1 & 2\\
 C_0 & 0 & 1 & -2
\end{array}$$

\subsubsection{Some nef divisors}\label{sssec-someNef}  Part of the nef cone of $X^{[n]}$ now follows from our knowledge of the intersection pairing.  First observe that since $C_0.D^{[n]} = 0$ and $C_0.B <0$, the nef cone is contained in the half-space of divisors with nonpositive $B$-coefficient in terms of the Fogarty isomorphism.

If $D$ is an ample divisor on $X$, then the divisor $D^{(n)}$ on the symmetric product is also ample, so $D^{[n]}$ is nef.  Since a limit of nef divisors is nef, it follows that if $D$ is nef on $X$ then $D^{[n]}$ is nef on $X^{[n]}$.  Furtermore, if $D$ is on the boundary of the nef cone of $X$ then $D^{[n]}$ is on the boundary of the nef cone of $X^{[n]}$.  Indeed, if $C.D = 0$ then $\tilde C_{[n]}.D^{[n]} = 0$ as well.  This proves
$$\Nef (X^{[n]}) \cap N^1(X) = \Nef(X),$$ where by abuse of notation we embed $N^1(X)$ in $N^1(X^{[n]})$ by $D\mapsto D^{[n]}$.

Boundary nef divisors which are not contained in the hyperplane $N^1(X)$ are more interesting and more challenging to compute.  Bridgeland stability and the positivity lemma will give us a tool for computing and describing these classes.

\subsubsection{Examples}  We close our initial discussion of the birational geometry of Hilbert schemes of points by considering several examples from this classical point of view.

\begin{example}[$\P^{2[n]}$]\label{ex-P2nNef}
The Neron-Severi space $N^1(\P^{2[n]})$ of the Hilbert scheme of $n$ points in $\P^2$ is spanned by $H^{[n]}$ and $B$, where $H$ is the class of a line in $\P^2$.  Any divisior in the cone $(H,B]$ is negative on $C_0$, so the locus $B$ swept out by curves of class $C_0$ is contained in the stable base locus of any divisor in this chamber.  Since $B.\tilde H_{[n]}=0$ and $\tilde H_{[n]}$ is the class of a moving curve, the divisor $B$ is an extremal effective divisor.

The divisor $H^{[n]}$ is an extremal nef divisor by \S\ref{sssec-someNef}, so to compute the full nef cone we only need one more extremal nef class.
The line bundle $\OO_{\P^2}(n-1)$ is \emph{$n$-very ample}, meaning that if $Z\subset \P^2$ is any zero-dimensional subscheme of length $n$, then $H^0(I_Z(n-1))$ has codimension $n$ in $H^0(\OO_{\P^2}(n-1))$.  Consequently, if $G$ is the Grassmannian of codimension-$n$ planes in $H^0(\OO_{\P^2}(n-1))$, then the natural map $\phi: \P^{2[n]}\to G$ is a \emph{morphism}.  Thus $\phi^* \OO_G(1)$ is nef.  In our notation for divisors, putting $D = (n-1)H$ we conclude that $$\tilde D^{[n]} = (n-1)H^{[n]} - \frac{1}{2}B$$ is nef.  

Furthermore, $\tilde D^{[n]}$ is not ample.  Numerically, simply observe that $\tilde D^{[n]}. H_{[n]} = 0$.  More geometrically, if two length $n$ schemes $Z,Z'$ are contained in the same line $L$ then the subspaces $H^0(I_Z(n-1))$ and $H^0(I_{Z'}(n-1))$ are equal, so $\phi$ identifies $Z$ and $Z'$.  Note that if $Z$ and $Z'$ are both contained in a single line $L$ then their ideal sheaves can be written as extensions $$0\to \OO_{\P^2}(-1) \to I_Z \to \OO_L(-n)\to 0$$
$$0\to \OO_{\P^2}(-1) \to I_{Z'} \to \OO_L(-n)\to 0.$$ 
This suggests that if we have some new notion of semistability where $I_Z$ is strictly semistable with Jordan-H\"older factors $\OO_{\P^2}(-1)$ and $\OO_L(-n)$ then the ideal sheaves $I_Z$ and $I_{Z'}$ will be $S$-equivalent.  Thus, in the moduli space of such objects, $I_Z$ and $I_{Z'}$ will be represented by the same point of the moduli space.
\end{example}

\begin{example}[$\P^{2[2]}$]\label{ex-P2Hilb2}
The divisor $\tilde H^{[2]} = H^{[2]}-\frac{1}{2}B$ spanning an edge of the nef cone is also an extremal effective divisor on $\P^{2[2]}$.  Indeed, the orthogonal curve class $H_{[2]}$ is a moving curve on $\P^{2[2]}$.  Thus there two chambers in the stable base locus decomposition of $\overline{\Eff}(\P^{2[2]})$.
\end{example}

\begin{example}[$\P^{2[3]}$]\label{ex-P2Hilb3}
By Example \ref{ex-P2nNef}, on $\P^{2[3]}$ the divisor $2H^{[3]}-\frac{1}{2}B$ is an extremal nef divisor.  The open chambers of the stable base locus decomposition are $$(H^{[3]},B), (2H^{[3]}-\frac{1}{2}B,H^{[3]}), \mbox{ and } (H^{[3]}-\frac{1}{2}B,2H^{[3]}-\frac{1}{2}B).$$
To establish this, first observe that $H^{[3]} - \frac{1}{2}B$ is the class of the locus $D$ of collinear schemes, since $D.C_0 = 1$ and $D.\tilde H_{[3]}=1.$ The divisor $2H^{[3]}-\frac{1}{2}B$ is orthogonal to curves of class $H_{[3]}$, so the locus of collinear schemes swept out by these curves lies in the stable base locus of any divisor in $(H^{[3]}-\frac{1}{2}B, 2H^{[3]}-\frac{1}{2}B)$.  In the other direction, any divisor in $(H^{[3]}-\frac{1}{2}B,2H^{[3]}-\frac{1}{2}B)$ is the sum of a divisor on the ray spanned by $D$ and an ample divisor.  It follows that the stable base locus in this chamber is exactly $D$.
\end{example}

For many more examples of the stable base locus decomposition of $\P^{2[n]}$, see \cite{ABCH} for explicit examples with $n\leq 9$, \cite{CH} for a discussion of the chambers where monomial schemes are in the base locus, and \cite{HuizengaJAG,CHW} for the effective cone.  Alternately, see \cite{CHGokova} for a deeper survey.   Also, see the work of Li and Zhao \cite{LiZhao} for more recent developments unifying several of these topics.

\subsection{Birational geometry of moduli spaces of sheaves}

We now discuss some of the basic aspects of the birational geometry of moduli spaces of sheaves.  Many of the concepts are mild generalizations of the picture for Hilbert schemes of points.

\subsubsection{Line bundles}\label{ssec-lineBundles} The main method of constructing line bundles on a moduli space of sheaves is by a determinantal construction.  First suppose $\cE/S$ is a family of sheaves on $X$ parameterized by $S$.  Let $p:S\times X\to S$ and $q:S\times X\to X$ be the projections.  The \emph{Donaldson homomorphism} is a map $\lambda_\cE:K(X)\to \Pic(S)$ defined by the composition $$\lambda_\cE:K(X)\fto{q^*} K^0(S\times X)\fto{\cdot [\cE]} K^0(S\times X) \fto{p_!} K^0(S)\fto{\det}\Pic(S) $$ Here $p_! = \sum_i (-1)^i R^ip_*$.  Informally, we pull back a sheaf on $X$ to the product, twist by the family $\cE$, push forward to $S$, and take the determinant line bundle.  Thus we obtain from any class in $K(X)$ a line bundle on the base $S$ of the family $\cE$. The above discussion is sufficient to define line bundles on a moduli space $M({\bf v})$ of sheaves if there is a universal family $\cE$ on $M({\bf v})$: there is then a map $\lambda_\cE:K(X)\to \Pic(M({\bf v}))$, and the image typically consists of many interesting line bundles on the moduli space.  

Things are slightly more delicate in the general case where there is no universal family. As motivation, given a class ${\bf w}\in K(X)$, we would like to define a line bundle $L$ on $M({\bf v})$ with the following property.  Suppose $\cE/S$ is a family of sheaves of character ${\bf v}$ and that $\phi:S\to M({\bf v})$ is the moduli map.  Then we would like there to be an isomorphism $\phi^*L \cong \lambda_{\cE}({\bf w}),$ so that the determinantal line bundle $\lambda_{\cE}({\bf w})$ on $S$ is the pullback of a line bundle on the moduli space $M({\bf v})$.

In order for this to be possible, observe that the line bundle $\lambda_{\cE}({\bf w})$ must be unchanged when it is replaced by $\cE\te p^*N$ for some line bundle $N\in \Pic(S)$.  Indeed, the moduli map $\phi: S\to M({\bf v})$ is not changed when we replace $\cE$ by $\cE\te p^*N$, so $\phi^*L$ is unchanged as well.  However, a computation shows that $$\lambda_{\cE\te p^*N}({\bf w}) =  \lambda_{\cE}({\bf w}) \te N^{\otimes \chi({\bf v}\te {\bf w})}.$$ Thus, in order for there to be a chance of defining a line bundle $L$ on $M({\bf v})$ with the desired property we need to assume that $\chi({\bf v}\te {\bf w})=0$.

In fact, if $\chi({\bf v}\te {\bf w})=0$, then there is a line bundle $L$ as above on the stable locus $M^s({\bf v})$, denoted by $\lambda^s({\bf w})$. To handle things rigorously, it is necessary to go back to the construction of the moduli space via GIT.  See \cite[\S 8.1]{HuybrechtsLehn} for full details, as well as a discussion of line bundles on the full moduli space $M({\bf v})$.

\begin{theorem}[{\cite[Theorem8.1.5]{HuybrechtsLehn}}]
Let ${\bf v}^\perp\subset K(X)$ denote the orthogonal complement of ${\bf v}$ with respect to the Euler pairing $\chi(-\otimes -)$.  Then there is a natural homomorphism $$\lambda^s: {\bf v^\perp}\to \Pic(M^s({\bf v})).$$
\end{theorem}

In general it is a difficult question to completely determine the Picard group of the moduli space.  One of the best results in this direction is the following theorem of Jun Li.

\begin{theorem}[\cite{JunLiPicard}]
Let $X$ be a regular surface, and let ${\bf v}\in K(X)$ with $\rk {\bf v}=2$ and $\Delta({\bf v}) \gg 0$.  Then the map $$\lambda^s : {\bf v}^\perp \te \Q \to \Pic(M^s({\bf v}))\te \Q$$ is a surjection.
\end{theorem}

More precise results are somewhat rare.  We discuss a few of the main such examples here.

\begin{example}[Picard group of moduli spaces of sheaves on $\P^2$]
Let $M({\bf v})$ be a moduli space of sheaves on $\P^2$.  The Picard group of this space was determined by Dr\'ezet \cite{DrezetPicard}.  The answer depends on the $\delta$-function introduced in the classification of semistable characters in \S\ref{sssec-existP2}.  If ${\bf v}$ is the character of an exceptional bundle then $M({\bf v})$ is a point and there is nothing to discuss.  If $\delta(\mu({\bf v})) = \Delta({\bf v})$, then $M({\bf v})$ is a moduli space of so-called \emph{height zero} bundles and the Picard group is isomorphic to $\Z$.  Finally, if $\delta(\mu({\bf v}))> \Delta({\bf v})$ then the Picard group is isomorphic to $\Z\oplus \Z$.  In each case, the Donaldson morphism is surjective.
\end{example}

\begin{example}[Picard group of moduli spaces of sheaves on $\P^1\times \P^1$] Let $M({\bf v})$ be a moduli space of sheaves on $\P^1\times \P^1$.  Already in this case the Picard group does not appear to be known in every case.  See \cite{YoshiokaRuled} for some partial results, as well as results on ruled surfaces in general.
\end{example}

\begin{example}[Picard group of moduli spaces of sheaves on a $K3$ surface] Let $X$ be a $K3$ surface, and let ${\bf v}\in K_{\num}(X)$ be a primitive positive vector (see \S\ref{sssec-K3exist}).  Let $H$ be a polarization which is generic with respect to ${\bf v}$.  In this case the story is similar to the computation for $\P^2$, with the Beauville-Bogomolov form playing the role of the $\delta$ function.  If $\langle{\bf v},{\bf v}\rangle = -2$ then $M_H({\bf v})$ is a point.  If $\langle {\bf v},{\bf v}\rangle = 0$, then the Donaldson morphism $\lambda: {\bf v}^\perp\te \R \to N^1(M_H({\bf v}))$ is surjective with kernel spanned by ${\bf v}$, and $N^1(M_H({\bf v}))$ is isomorphic to ${\bf v}^\perp/{\bf v}$.  Finally, if $\langle{\bf v},{\bf v}\rangle >0$ then the Donaldson morphism is an isomorphism.  See \cite{Yoshioka3} or \cite{BM2} for details.
\end{example}

\begin{example}[Brill-Noether divisors]
For birational geometry it is important to be able to construct sections of of line bundles.  The determinantal line bundles introduced above frequently have special sections vanishing on \emph{Brill-Noether divisors}.  Let $(X,H)$ be a smooth surface, and let ${\bf v}$ and ${\bf w}$ be an orthogonal pair of Chern characters, i.e. suppose that $\chi({\bf v}\te {\bf w})=0$, and suppose that there is a reasonable, e.g. irreducible, moduli space $M_H({\bf v})$ of semistable sheaves.  Suppose $F$ is a vector bundle with $\ch F = {\bf w}$, and consider the locus $$D_F = \{ E \in M_H({\bf v}): H^0(E\te F)\neq 0\}.$$ If we assume that $H^2(E\te F)=0$ for every $E\in M_H({\bf v})$ and that $H^0(E\te F) = 0$ for a general $E\in M_H({\bf v})$ then the locus $D_F$ will be an effective divisor.  Furthermore, its class is $\lambda({\bf w}^*)$.

The assumption that $H^2(E\te F) = 0$ often follows easily from stability and Serre duality.  For instance, if $\mu_H({\bf v}),\mu_H({\bf w}),\mu_H(K_X)> 0$ and $F$ is a semistable vector bundle then $$H^2(E\te F) = \Ext^2(F^*,E) = \Hom(E,F^*(K_X))^*=0$$ by stability.  On the other hand, it can be quite challenging to verify that $H^0(E\te F)= 0$ for a general $E\in M_H({\bf v})$.  These types of questions have been studied in \cite{CHW} in the case of $\P^2$ and \cite{Ryan} in the case of $\P^1\times \P^1$.  Interesting effective divisors arising in the birational geometry of moduli spaces frequently arise in this way.
\end{example}

\subsubsection{The Donaldson-Uhlenbeck-Yau compactification} For Hilbert schemes of points $X^{[n]}$, the symmetric product $X^{(n)}$ offered an alternate compactification, with the map $h:X^{[n]}\to X^{(n)}$ being the Hilbert-Chow morphism.  Recall that from a moduli perspective the Hilbert-Chow morphism sends the ideal sheaf $I_Z$ to (the $S$-equivalence class of) the structure sheaf $\OO_Z$.  Thinking of $\OO_X$ as the double-dual of $I_Z$, the sheaf $\OO_Z$ is the cokernel in the sequence $$0\to I_Z\to \OO_X\to \OO_Z\to 0.$$ The Donaldson-Uhlenbeck-Yau compactification can be viewed as analogous to the compactification of the Hilbert scheme by the symmetric product. 

Let $(X,H)$ be a smooth surface, and let ${\bf v}$ be the Chern character of a semistable sheaf of positive rank.  Set-theoretically, the Donaldson-Uhlenbeck-Yau compactification $M_H^{DUY}({\bf v})$ of the moduli space $M_H({\bf v})$ can be defined as follows.  Recall that the double dual of any torsion-free sheaf $E$ on $X$ is locally free, and there is a canonical inclusion $E\to E^{**}$.  (Note, however, that the double-dual of a Gieseker semistable sheaf is in general only $\mu_H$-semistable).  Define $T_E$ as the cokernel $$0\to E \to E^{**}\to T_E\to 0,$$ so that $T_E$ is a skyscraper sheaf supported on the singularities of $E$.  In the Donaldson-Uhlenbeck-Yau compactification of $M_H({\bf v})$, a sheaf $E$ is replaced by the pair $(E^{**},T_E)$ consisting of the $\mu_H$-semistable sheaf $E^{**}$ and the $S$-equivalence class of $T_E$, i.e. an element of some symmetric product $X^{(n)}$.  In particular, two sheaves which have isomorphic double duals and have singularities supported at the same points (counting multiplicity) are identified in $M_H^{DUY}({\bf v})$, even if the particular singularities are different.  The Jun Li morphism $j:M_H({\bf v})\to M_H^{DUY}({\bf v})$ inducing the Donaldson-Uhlenbeck-Yau compactification arises from the line bundle $\lambda({\bf w})$ associated to the character ${\bf w}$ of a $1$-dimensional torsion sheaf supported on a curve whose class is a multiple of $H$.  See \cite[\S8.2]{HuybrechtsLehn} or \cite{JunLi} for more details.

\subsubsection{Change of polarization}\label{sssec-variation} Classically, one of the main interesting sources of birational maps between moduli spaces of sheaves is provided by varying the polarization.  Suppose that $\{H_t\}$ $(0\leq t \leq 1)$ is a continuous family of ample divisors on $X$.  Let $E$ be a sheaf which is $\mu_{H_0}$-stable.  It may happen for some time $t>0$ that $E$ is not $\mu_{H_t}$-stable.  In this case, there is a smallest time $t_0$ where $E$ is not $\mu_{H_{t_0}}$-stable, and then $E$ is strictly $\mu_{H_{t_0}}$-semistable.  There is then an exact sequence $$0\to F \to E \to G\to 0$$ of $\mu_{H_{t_0}}$-semistable sheaves with the same $\mu_{H_{t_0}}$-slope.  For $t<t_0$, we have $$\mu_{H_t}(F)<\mu_{H_t}(E)<\mu_{H_t}(G).$$ On the other hand, in typical examples the inequalities will be reversed for $t>t_0$: $$\mu_{H_t}(F)>\mu_{H_t}(E)>\mu_{H_t}(G).$$ While $E$ is certainly not $\mu_{H_t}$-semistable for $t>t_0$, if there are sheaves $E'$ fitting as extensions in sequences $$0\to G\to E'\to F\to 0$$ then it may happen that $E'$ is $\mu_{H_t}$-stable for $t>t_0$ (although they are certainly not $\mu_{H_t}$-semistable for $t<t_0$).

Thus, the set of $H_t$-semistable sheaves changes as $t$ crosses $t_0$, and the moduli space $M_{H_t}({\bf v})$ changes accordingly.  It frequently happens that only some very special sheaves become destabilized as $t$ crosses $t_0$, in which case the expectation would be that the moduli spaces for $t<t_0$ and $t>t_0$ are birational.

To clarify the dependence between the geometry of the moduli space $M_H({\bf v})$ and the choice of polarization $H$, we partition the cone $\Amp(X)$ of ample divisors on $X$ into chambers where the moduli space remains constant. Let ${\bf v}$ be a primitive vector, and suppose $E$ has $\ch(E) = {\bf v}$ and is strictly $H$-semistable for some polarization $H$.  Let $F\subset E$ be an $H$-semistable subsheaf with $\mu_H(F) = \mu_H(E)$.  Then the locus $\Lambda \subset \Amp(X)$ of polarizations $H'$ such that $\mu_{H'}(F) = \mu_{H'}(E)$ is a hyperplane in the ample cone, called a \emph{wall}.  The collection of all walls obtained in this way gives the ample cone a locally finite wall-and-chamber decomposition.  As $H$ varies within an open chamber, the moduli space $M_H({\bf v})$ remains unchanged.  On the other hand, if $H$ crosses a wall then the moduli spaces on either side may be related in interesting ways.   

Notice that if say $X$ has Picard rank $1$ or we are considering Hilbert schemes of points then no interesting geometry can be obtained by varying the polarization.  Recall that in Example \ref{ex-P2Hilb3} we saw that even $\P^{2[3]}$ has nontrivial alternate birational models.  One of the goals of Bridgeland stability will be to view these alternate models as a variation of the stability condition.  Variation of polarization is one of the simplest examples of how a stability condition can be modified in a continuous way, and Bridgeland stability will give us additional ``degrees of freedom'' with which to vary our stability condition. 

\section{Bridgeland stability}\label{sec-Bridgeland}

The definition of a Bridgeland stability condition needs somewhat more machinery than the previous sections.  However, we will primarily work with explicit stability conditions where the abstract nature of the definition becomes very concrete.  While it would be a good idea to review the basics of derived categories of coherent sheaves, triangulated categories, t-structures, and torsion theories, it is also possible to first develop an appreciation for stability conditions and then go back and fill in the missing details.  Good references for background on these topics include \cite{GelfandManin} and \cite{Huybrechts}. 

\subsection{Stability conditions in general} Let $X$ be a smooth projective variety. We write $D^b(X)$ for the bounded derived category of coherent sheaves on $X$.  We also write $K_{\num}(X)$ for the Grothendieck group of $X$ modulo numerical equivalence.  Following \cite{bridgeland:stable}, we make the following definition.

\begin{definition}\label{def-Bridgeland}
A \emph{Bridgeland stability condition} on $X$ is a pair $\sigma =(Z,\cA)$ consisting of an $\R$-linear map $Z:K_{\num}(X)\te \R\to \C$ (called the \emph{central charge}) and the heart $\cA\subset D^b(X)$ of a bounded t-structure (which is an abelian category).  Additionally, we require that the following properties be satisfied.
\begin{enumerate}
\item\label{ax-positivity} (Positivity) If $0\neq E\in \cA$, then $$Z(E)\in \HH:=\{re^{i\theta}:0 < \theta \leq \pi \textrm{ and } r> 0\}\subset \C.$$ We define functions $r(E) = \Im Z(E)$  and $d(E) = -\Re Z(E)$, so that $r(E)\geq 0$ and $d(E) >0$ whenever $r(E)=0$.  Thus $r$ and $d$ are generalizations of the classical rank and degree functions.  The \emph{(Bridgeland) $\sigma$-slope} is defined by $$\mu_\sigma(E) = \frac{d(E)}{r(E)} = -\frac{\Re Z(E)}{\Im Z(E)}.$$
\item (Harder-Narasimhan filtrations) An object $E\in \cA$ is called (Bridgeland) $\sigma$-(semi)stable if $$\mu_\sigma(F) \leqpar \mu_\sigma(E)$$ whenever $F\subset E$ is a subobject of $E$ in $\cA$.  We require that every object of $\cA$ has a finite Harder-Narasimhan filtration in $\cA$.  That is, there is a unique filtration $$0= E_0 \subset E_1 \subset \cdots \subset E_\ell = E$$ of objects $E_i\in \cA$ such that the quotients $F_i = E_i/E_{i-1}$ are $\sigma$-semistable with decreasing slopes $\mu_\sigma(F_1) > \cdots > \mu_{\sigma}(F_\ell)$.
\item (Support property) The support property is one final more technical condition which must be satisfied.  Fix a norm $\|\cdot\|$ on $K_{\num}(X)\te \R$.  Then there must exist a constant $C>0$ such that $$\|E\| \leq C \| Z(E)\|$$ for all semistable objects $E\in \cA$.
\end{enumerate}
\end{definition}

\begin{remark}
Let $(X,H)$ be a smooth surface.  The subcategory $\coh X\subset D^b(X)$ of sheaves with cohomology supported in degree $0$ is the heart of the standard t-structure. We can then try to define a central charge $$Z(E) = -c_1(E).H+i\rk(E)H^2,$$ and the corresponding slope function is the ordinary slope $\mu_H$.  However, this does \emph{not} give a Bridgeland stability condition, since $Z(E)=0$ for any finite length torsion sheaf.  Thus it is not immediately clear in what way Bridgeland stability generalizes ordinary slope- or Gieseker stability.  Nonetheless, for any fixed polarization $H$ and character ${\bf v}$ there are Bridgeland stability conditions $\sigma$ where the $\sigma$-(semi)stable objects of character ${\bf v}$ are precisely the $H$-Gieseker (semi)stable sheaves of character ${\bf v}$.  See \S\ref{ssec-largeVolume} for more details.
\end{remark}

\begin{remark}
To work with the definition of a stability condition, it is crucial to understand what it means for a map $F\to E$ between objects of the heart $\cA$ to be injective.  The following exercise is a good test of the definitions involved.
\begin{exercise}\label{ex-exact}
Let $\cA \subset D^b(X)$ be the heart of a bounded t-structure, and let $\phi:F\to E$ be a map of objects of $\cA$.  Show that $\phi$ is injective if and only if the mapping cone $\cone(\phi)$ of $\phi$ is also in $\cA$.  In this case, there is an exact sequence $$0\to F\to E\to \cone(\phi)\to 0$$ in $\cA$.
\end{exercise}
\end{remark}

One of the most important features of Bridgeland stability is that the space of all stability conditions on $X$ is a complex manifold in a natural way.  In particular, we are able to continuously vary stability conditions and study how the set (or moduli space) of semistable objects varies with the stability condition.  Let $\Stab(X)$ denote the space of stability conditions on $X$.  Then Bridgeland proves that there is a natural topology on $\Stab(X)$ such that the forgetful map \begin{align*}
\Stab(X)&\to \Hom_\R(K_{\num}(X)\te \R, \C)\\
(Z,\cA) & \mapsto Z
\end{align*}
is a local homeomorphism.  Thus if $\sigma = (Z,\cA)$ is a stability condition and the linear map $Z$ is deformed by a small amount, there is a unique way to deform the category $\cA$ to get a new stability condition.

\subsubsection{Moduli spaces} Let $\sigma$ be a stability condition and fix a vector ${\bf v}\in K_{\num}(X)$.  There is a notion of a flat family $\cE/S$ of $\sigma$-semistable objects parameterized by an algebraic space $S$ \cite{BM2}.  Correspondingly, there is a moduli stack $\cM_\sigma({\bf v})$ parameterizing flat families of $\sigma$-semistable object of character ${\bf v}$.  In full generality there are many open questions about the geometry of these moduli spaces.  In particular, when is there a projective coarse moduli space $M_{\sigma}({\bf v})$ parameterizing $S$-equivalence classes of $\sigma$-semistable objects of character ${\bf v}$?  

Several authors have addressed this question for various surfaces, at least when the stability condition $\sigma$ does not lie on a \emph{wall} for ${\bf v}$ (see \S \ref{ssec-walls}).  For instance, there is a projective moduli space $M_\sigma({\bf v})$ when $X$ is $\P^2$ \cite{ABCH}, $\P^1\times \P^1$ or $\F_1$ \cite{ArcaraMiles}, an abelian surface \cite{MYY}, a  $K3$ surface \cite{BM2}, or an Enriques surface \cite{Nuer}.  While projectivity of Gieseker moduli spaces can be shown in great generality, there is no known uniform GIT construction of moduli spaces of Bridgeland semistable objects.  Each proof requires deep knowledge of the particular surface.

\subsection{Stability conditions on surfaces}\label{ssec-surfaceConds}
Bridgeland \cite{bridgelandK3} and Arcara-Bertram \cite{AB} explain how to construct stability conditions on a smooth surface.  The construction is very explicit, and these are the only kinds of stability conditions we will consider in this survey.  Before beginning we introduce some notation to make the definitions more succinct.

Let $X$ be a smooth surface and let $H,D\in \Pic(X)\te \R$ be an ample divisor and an arbitrary \emph{twisting} divisor, respectively.  We formally define the \emph{twisted Chern character} $\ch^D = e^{-D} \ch$.  Explicitly expanding this definition, this means that 
\begin{align*}
\ch^D_0 &= \ch_0\\
\ch^D_1 &= \ch_1 - D\ch_0\\
\ch^D_2 &= \ch_2 - D\ch_1 + \frac{D^2}{2} \ch_0.
\end{align*}
We can also define twisted slopes and discriminants by the formulas
\begin{align*}
\mu_{H,D} &= \frac{H.\ch_1^D}{H^2 \ch_0^D}\\
\Delta_{H,D} &= \frac{1}{2} \mu_{H,D}^2 - \frac{\ch_2^D}{H^2\ch_0^D}.
\end{align*}
For reasons that will become clear in \S\ref{ssec-largeVolume} it is often useful to add in an additional twist by $K_X/2$.  We therefore additionally define $$\overline\ch^D = \ch^{D+\frac{1}{2}K_X} \qquad \overline \mu_{H,D} = \mu_{H,D+\frac{1}{2}K_X} \qquad \overline\Delta_{H,D} = \Delta_{H,D+\frac{1}{2}K_X}.$$ \begin{remark}
Note that the twisted slopes $\mu_{H,D}$ and $\overline \mu_{H,D}$ are primarily just a notational convenience; they only differ from the ordinary slope by a constant (depending on $H$ and $D$).  On the other hand, twisted discriminants $\Delta_{H,D}$ and $\overline\Delta_{H,D}$ do not obey such a simple formula, and are genuinely useful.
\end{remark}
\begin{remark}[Twisted Gieseker stability]
We have already encountered $H$-Gieseker (semi)stability and the associated moduli spaces $M_H({\bf v})$ of $H$-Gieseker semistable sheaves.  There is a mild generalization of this notion called \emph{$(H,D)$-twisted Gieseker (semi)stability}.  A torsion-free coherent sheaf $E$ is $(H,D)$-twisted Gieseker (semi)stable if whenever $F\subsetneq E$ we have 
\begin{enumerate}
\item $\overline \mu_{H,D} (F)\leq \overline \mu_{H,D}(E)$ and
\item whenever $\overline \mu_{H,D}(F) = \overline\mu_{H,D}(E)$, we have $\overline\Delta_{H,D}(F) \geqpar \overline\Delta_{H,D}(E)$.
\end{enumerate}
Compare with Example \ref{ex-stabilitySurface}, which is the case $D=0$.  When $H,D$ are $\Q$-divisors, Matsuki and Wentworth \cite{MatsukiWentworth} construct projective moduli spaces $M_{H,D}({\bf v})$ of $(H,D)$-twisted Gieseker semistable sheaves.  Note that any $\mu_H$-stable sheaf is both $H$-Gieseker stable and $(H,D)$-twisted Gieseker stable, so that the spaces $M_H({\bf v})$ and $M_{H,D}({\bf v})$ are often either isomorphic or birational.
\end{remark}

\begin{exercise}\label{ex-twistedBogomolov}
Use the Hodge Index Theorem and the ordinary Bogomolov inequality (Theorem \ref{thm-bogomolov}) to show that if $E$ is $\mu_H$-semistable then $$\overline\Delta_{H,D}(E) \geq 0.$$
\end{exercise}

We now define a half-plane (or \emph{slice}) of stability conditions on $X$ corresponding to a choice of divisors $H,D\in \Pic(X)\te \R$ as above.  First fix a number $\beta\in \R$.  We define two full subcategories of the category $\coh X$ of coherent sheaves by 
\begin{align*}
\cT_\beta &= \{E\in \coh X:\overline \mu_{H,D} (G) > \beta \textrm{ for every quotient $G$ of }E\}\\
\cF_\beta &= \{E \in \coh X: \overline\mu_{H,D}(F) \leq \beta \textrm{ for every subsheaf $F$ of }E\}.
\end{align*}
Note that by convention the (twisted) Mumford slope of a torsion sheaf is $\infty$, so that $\cT_\beta$ contains all the torsion sheaves on $X$.  On the other hand, sheaves in $\cF_\beta$ have no torsion subsheaf and so are torsion-free.

For any $\beta\in \R$, the pair of categories $(\cT_\beta,\cF_\beta)$ form what is called a \emph{torsion pair}.  Briefly, this means that $\Hom(T,F)=0$ for any $T\in \cT_\beta$ and $F\in \cF_\beta$, and any $E\in \coh X$ can be expressed naturally as an extension $$0\to F\to E\to T\to 0$$ of a sheaf $T\in T_\beta$ by a sheaf $F\in \cF_\beta$.  Then there is an associated t-structure with heart $$\cA_\beta = \{E^{\bullet} : \rH^{-1}(E^{\bullet})\in \cF_\beta,\rH^0(E^\bullet) \in \cT_{\beta}, \textrm{ and } \rH^i(E^\bullet)=0\textrm{ for } i\neq -1,0\} \subset D^b(X),$$ where we use a Roman $\rH^i(E^{\bullet})$ to denote cohomology sheaves.

 Some objects of $\cA_\beta$ are the sheaves $T$ in $\cT_\beta$ (viewed as complexes sitting in degree $0$) and shifts $F[1]$ where $F\in \cF_\beta$, sitting in degree $-1$.  More generally, every object $E^{\bullet}\in \cA_\beta$ is an extension $$0\to \rH^{-1}(E^{\bullet})[1] \to E^\bullet \to \rH^0(E^{\bullet})\to 0,$$ where the sequence is exact in the heart $\cA_\beta$.

To define stability conditions we now need to define central charges compatible with the hearts $\cA_\beta$.  Let $\alpha\in \R_{>0}$ be an arbitrary positive real number.  We define $$Z_{\beta,\alpha} = -\overline \ch_2^{D+\beta H}+\frac{\alpha^2H^2}{2}\overline \ch_0^{D+\beta H} + i H \overline\ch_1^{D+\beta H},$$ and put $\sigma_{\beta,\alpha} = (Z_{\beta,\alpha},\cA_\beta)$.  Note that if $E$ is an object of nonzero rank with twisted slope $\overline \mu_{H,D}$ and discriminant $\overline \Delta_{H,D}$ then the corresponding Bridgeland slope is $$\mu_{\sigma_{\beta,\alpha}} = -\frac{\Re Z_{\beta,\alpha}}{\Im Z_{\beta,\alpha}} = \frac{(\overline \mu_{H,D} - \beta)^2 -\alpha^2 - 2\overline \Delta_{H,D}}{\overline \mu_{H,D} - \beta}.$$
\begin{theorem}[\cite{AB}]
Let $X$ be a smooth surface, and let $H,D\in \Pic(X)\te \R$ with $H$ ample.  If $\beta,\alpha\in \R$ with $\alpha>0$, then the pair $\sigma_{\beta,\alpha}=(Z_{\beta,\alpha},\cA_\beta)$ defined above is a Bridgeland stability condition.
\end{theorem}
The most interesting part of the theorem is the verification of the Positivity axiom \ref{ax-positivity} in the Definition \ref{def-Bridgeland} of a stability condition, which we now sketch.  The other parts are quite formal.
\begin{proof}[Sketch proof of positivity]
Note that $Z:=Z_{\beta,\alpha}$ is an $\R$-linear map.  Since the upper half-plane $\HH = \{re^{i\theta}:0< \theta \leq \pi \textrm{ and } r>0\}$ is closed under addition, the exact sequence $$0\to \rH^{-1}(E^{\bullet})[1] \to E^{\bullet} \to \rH^0(E^\bullet)\to 0$$ implies that it is sufficient to check $Z(T)\in \HH$ and $Z(F[1]) \in \HH$ whenever $T\in \cT_\beta$ and $F\in \cF_\beta$.  

If $T\in \cT_\beta$ is not torsion, then $\overline\mu_{H,D}(T) > \beta$ is finite.  Expanding the definitions immediately gives $H.\overline\ch_1^{D+\beta H}(T)>0$, so $Z(T)\in \HH$.  If $T$ is torsion with positive-dimensional support, then again $H.\overline\ch_1^{D+\beta H}(T) >0$ and $Z(T)\in \HH$.  Finally, if $T\neq 0$ has zero-dimensional support then $-\overline \ch_2^{D+\beta H}(T) = -\ch_2(T) >0$ so $Z(T)\in \HH$.

Suppose $0\neq F\in \cF_\beta$.  If actually $\overline\mu_{H,D}(F)<\beta$, then $H.\overline\ch_1^{D+\beta H}(F) < 0$ and $Z(F[1])\in \HH$ again follows.  So suppose that $\overline \mu_{H,D}(F) = \beta$, which gives $\Im Z(F)=0$.  By the definition of $\cF_\beta$, the sheaf $F$ is torsion-free and $\overline \mu_{H,D+\beta H}$-semistable of $\overline \mu_{H,D+\beta H}$ slope $0$. By Exercise \ref{ex-twistedBogomolov} we find that $\overline \Delta_{H,D+\beta H}(F) \geq 0$.  The formula for the twisted discriminant and the fact that $\alpha>0$ then gives $\Re Z(F)<0$, so $\Re Z(F[1]) >0$.      
\end{proof}

To summarize, if we let $\Pi=\{(\beta,\alpha):\beta,\alpha\in \R, \alpha>0\}$, the choice of a pair of divisors $H,D\in \Pic(X)\te \R$ with $H$ ample defines an embedding
\begin{align*}
\Pi &\to \Stab(X)\\
(\beta,\alpha)& \mapsto \sigma_{\beta,\alpha}.
\end{align*}
This half-plane of stability conditions is called the \emph{$(H,D)$-slice} of the stability manifold.  We will sometimes abuse notation and write $\sigma\in \Pi$ for a stability condition $\sigma$ parameterized by the slice.  While the stability manifold can be rather large and unwieldy in general (having complex dimension $\dim_\R K_{\num}(X)\te \R$), much of the interesting geometry can be studied by inspecting the different slices of the manifold.

\subsection{Walls}\label{ssec-walls}
Fix a class ${\bf v}\in K_{0}(X)$.  The stability manifold $\Stab(X)$ of $X$ admits a locally finite wall-and-chamber decomposition such that the set of $\sigma$-semistable objects of class ${\bf v}$ does not vary as $\sigma$ varies within an open chamber.  This is analogous to the wall-and-chamber decomposition of the ample cone $\Amp(X)$ for classical stability, see \S\ref{sssec-variation}.  If ${\bf v}$ is primitive, then a stability condition $\sigma$ lies on a wall if and only if there is a strictly $\sigma$-semistable object of character ${\bf v}$.  

For computations, the entire stability manifold can be rather unwieldy to work with.  One commonly restricts attention to stability conditions in some easily parameterized subset of the stability manifold.   
Here we focus on the $(H,D)$-slice $\{\sigma_{\beta,\alpha}:\beta,\alpha\in \R, \alpha>0\}$ of stability conditions on a smooth surface $X$ determined by a choice of divisors $H,D\in \Pic(X)\te \R$ with $H$ ample. 

\begin{definition}
Let $X$ be a smooth surface, and fix divisors $H,D\in \Pic(X)\te \R$ with $H$ ample.  Let ${\bf v}, {\bf w} \in K_{\num}(X)$ be two classes which have different $\mu_{\sigma_{\beta,\alpha}}$-slopes for some $(\beta,\alpha)$ with $\alpha>0$.
\begin{enumerate}
\item The \emph{numerical wall}  for $\bf v$ determined by $\bf w$ is the subset $$W({\bf v},{\bf w})=\{(\beta,\alpha):\mu_{\sigma_{\beta,\alpha}}({\bf v}) = \mu_{\sigma_{\beta,\alpha}}({\bf w})\}\subset \Pi.$$ 

\item The numerical wall for ${\bf v}$ determined by ${\bf w}$ is a \emph{wall} if there is some $(\beta,\alpha)\in W({\bf v},{\bf w})$  and an exact sequence $$0\to F\to E\to G\to 0$$ of $\sigma_{\beta,\alpha}$-semistable objects with $\ch F = {\bf w}$ and $\ch E = {\bf v}$.
\end{enumerate}
\end{definition}

\subsubsection{Geometry of numerical walls} The geometry of the numerical walls in a slice of the stability manifold is particularly easy to describe.  Verifying the following properties is a good exercise in the algebra of Chern classes and the Bridgeland slope function.
\begin{enumerate}
\item First suppose ${\bf v}$ has nonzero rank and that the Bogomolov inequality $\Delta_{H,D}({\bf v}) \geq 0$ holds.  Then the vertical line $\beta = \overline\mu_{H,D}({\bf v})$ is a numerical wall.  The other numerical walls form two nested families of semicircles on either side of the vertical wall.  These semicircles have centers on the $\beta$-axis, and their apexes lie along the hyperbola $\Re Z_{\beta,\alpha}({\bf v}) = 0$ in $\Pi$.  The two families of semicircles accumulate at the points $$(\mu_{H,D}({\bf v})\pm \sqrt{2\Delta_{H,D}({\bf v})},0)$$ of intersection of $\Re Z_{\beta,\alpha}({\bf v}) = 0$ with the $\beta$-axis.  See Figure \ref{fig-walls} for an approximate illustration.
\begin{figure}[htbp]
\begin{center}
\input{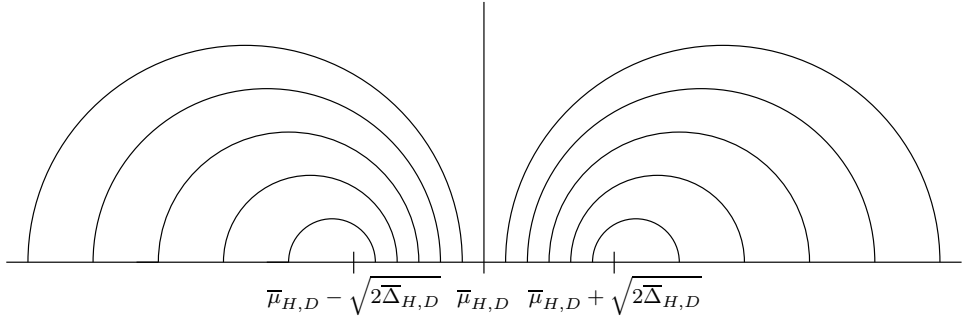}
\end{center}
\caption{Schematic diagram of numerical walls in the $(H,D)$-slice for a nonzero rank character ${\bf v}$ with slope $\OV \mu_{H,D}$ and discriminant $\OV \Delta_{H,D}$.}
\label{fig-walls}
\end{figure}   

\item If instead ${\bf v}$ has rank zero but $c_1({\bf v})\neq 0$, then the curve $\Re Z_{\beta,\alpha}({\bf v}) = 0$ in $\Pi$ degenerates to the vertical line $$\beta = \frac{\overline\ch_2^D({\bf v})}{\overline\ch_1^D({\bf v}).H}.$$ The numerical walls for ${\bf v}$ are all semicircles with center $(\overline\ch_2^D({\bf v})/(\overline\ch_1^D({\bf v}).H),0)$ and arbitrary radius.
\end{enumerate}

\begin{exercise}In ${\bf v},{\bf w}$ have nonzero rank and different slopes, the numerical semicircular wall $W({\bf v},{\bf w})$ has center $(s_W,0)$ and radius $\rho_W$ satisfying \begin{align*} s_W &= \frac{\overline \mu_{H,D}({\bf v})+\overline \mu_{H,D}({\bf w})}{2}-\frac{\overline\Delta_{H,D}({\bf v}) - \overline\Delta_{H,D}({\bf w})}{\overline\mu_{H,D}({\bf v}) - \overline \mu_{H,D}({\bf w})}\\ \rho_W^2 &= (s_W-\overline\mu_{H,D}({\bf v}))^2 - 2\overline \Delta_{H,D}({\bf v}).\end{align*}  If $(s_W-\mu_{H,D}({\bf v}))^2\leq 2\overline\Delta_{H,D}({\bf v})$, then the wall is empty.
\end{exercise}

\begin{remark}
Let ${\bf v}$ be a character of nonzero rank.  It follows from the above discussion that if $W,W'$ are numerical walls for ${\bf v}$ both lying left of the vertical wall $\beta = \overline\mu_{H,D}({\bf v})$ then $W$ is nested inside $W'$ if and only if $s_W > s_{W'}$, where the center of $W$ (resp. $W'$) is $(s_W,0)$ (resp. $(s_{W'},0)$).
\end{remark}

\subsubsection{Walls and destabilizing sequences} In the definition of a wall $W:=W({\bf v},{\bf w})$ for ${\bf v}$ determined by a character ${\bf w}$ we required that there is some point $(\beta,\alpha)\in W$ and a \emph{destabilizing} exact sequence $$0\to F\to E\to G\to 0$$ of $\sigma_{\beta,\alpha}$-semistable objects, where $\ch(E) = {\bf v}$ and $\ch(F) = {\bf w}$.  Note that since $(\beta,\alpha)\in W$ we in particular have $\mu_{\sigma_{\beta,\alpha}}({ F}) = \mu_{\sigma_{\beta,\alpha}}({E}) = \mu_{\sigma_{\beta,\alpha}}(G)$.  The above sequence is an exact sequence of objects of the categories $\cA_\beta$.
By the geometry of the numerical walls, the wall $W$ separates the slice $\Pi$ into two open regions $\Omega,\Omega'$.  Relabeling the regions if necessary, for $\sigma \in \Omega$ we have $\mu_\sigma(F) > \mu_\sigma(E)$.  Therefore $E$ is not $\sigma$-semistable for any $\sigma\in \Omega$.  On the other hand, $E$ \emph{may} be $\sigma$-semistable for $\sigma\in \Omega$; at least the subobject $F\subset E$ does not violate the semistability of $E$.

Our definition of a wall is perhaps somewhat unsatisfactory due to the dependence on picking some point $(\beta,\alpha)\in W$ where there is a destabilizing exact sequence as above.  The next result shows that this definition is equivalent to an a priori stronger definition which appears more natural.  Roughly speaking, destabilizing sequences ``persist'' along the entire wall.

\begin{proposition}[{\cite[Lemma 6.3]{ABCH}} for $\P^2$, \cite{Maciocia} in general]\label{prop-destSeq}
Suppose that $$0\to F\to E\to G\to 0$$ is an exact sequence of $\sigma_{\beta,\alpha}$-semistable objects of the same $\sigma_{\beta,\alpha}$-slope.  Put $\ch F = {\bf w}$ and $\ch E = {\bf v}$, and suppose ${\bf v}$ and ${\bf w}$ do not have the same slope everywhere in the $(H,D)$-slice.  Let $W = W({\bf v},{\bf w})$ be the wall defined by these characters.  If $(\beta',\alpha')\in W$ is any point on the wall, then the above exact sequence is an exact sequence of $\sigma_{\beta',\alpha'}$-semistable objects of the same $\sigma_{\beta',\alpha'}$-slope.

In particular, each of the objects $F,E,G$ appearing in the above sequence lie in the category $\cA_{\beta'}$. 
\end{proposition}

Note that the first part of the proposition is essentially equivalent to the final statement by Exercise \ref{ex-exact}.

\subsection{Large volume limit}\label{ssec-largeVolume} As mentioned earlier, (twisted) Gieseker moduli spaces of sheaves on surfaces can be recovered as certain moduli spaces of Bridgeland-semistable objects.  We say that an object $E^\bullet\in \cA_\beta$ is a \emph{sheaf} if it is isomorphic to a sheaf sitting in degree $0$.  We continue to work in an $(H,D)$-slice of stability conditions on a smooth surface $X$.

\begin{theorem}[{\cite[\S 6]{ABCH}} for $\P^2$, \cite{Maciocia} in general]\label{thm-largeVol}
Let ${\bf v}\in K_{\num}(X)$ be a character of positive rank with $\overline\Delta_{H,D}({\bf v})\geq 0$.  Let $\beta < \overline\mu_{H,D}({\bf v})$, and suppose $\alpha \gg 0$ (depending on ${\bf v}$).  Then an object $E^\bullet\in \cA_\beta$  is $\sigma_{\beta,\alpha}$-semistable if and only if it is an $(H,D)$-semistable sheaf.  
\end{theorem}
\begin{proof}
Since $\beta< \overline\mu_{H,D}({\bf v})$, the stability condition $\sigma_{\beta,\alpha}$ lies left of the vertical wall $\beta = \overline\mu_{H,D}({\bf v})$.  The walls for ${\bf v}$ are locally finite.  Considering a neighborhood of a stability condition on the vertical wall shows that there is some largest semicircular wall $W$ left of the vertical wall.  The set of $\sigma$-semistable objects is constant as $\sigma$ varies in the chamber between $W$ and the vertical wall.

It is therefore enough to show the following two things.  (1) If $E^\bullet \in \cA_\beta$ has $\ch E^{\bullet} = {\bf v}$ and is $\sigma_{\beta,\alpha}$-semistable for $\alpha \gg 0$ then $E^{\bullet}$ is an $(H,D)$-semistable sheaf.  (2) If $E$ is an $(H,D)$-semistable sheaf of character ${\bf v}$, then $E$ is $\sigma_{\beta,\alpha}$-semistable for $\alpha \gg 0$.  That is, we may pick $\alpha$ depending on $E$, and not just depending on ${\bf v}$.

(1) First suppose $E^\bullet\in \cA_\beta$ is $\sigma_{\beta,\alpha}$-semistable for $\alpha\gg 0$ and $\ch E^{\bullet} = {\bf v}$. If $E^\bullet$ is not a sheaf, then we have an interesting exact sequence $$0\to \rH^{-1}(E^{\bullet})[1] \to E^{\bullet} \to \rH^0(E^{\bullet})\to 0$$ in $\cA_\beta$.  Since $F:=\rH^{-1}(E^\bullet)\in \cF_\beta$, the formula for the Bridgeland slope shows that $$\mu_{\sigma_{\beta,\alpha}}(F[1])=\mu_{\sigma_{\beta,\alpha}}(F)\to \infty$$ as $\alpha\to \infty$.  On the other hand, since $G:= H^0(E^\bullet)\in \cT_\beta$ we have $\mu_{\sigma_{\beta,\alpha}}(G)\to -\infty$ as $\alpha\to \infty$, noting that $\rk(G) > 0$ since $\rk {\bf v} >0$.  This is absurd since $E^\bullet$ is $\sigma_{\beta,\alpha}$-semistable for $\alpha \gg 0$, and we conclude that $E:=E^\bullet\in \cT_\beta$ is a sheaf.

Similar arguments show that $E$ is $(H,D)$-semistable.  First suppose $E$ has a $\overline\mu_{H,D}$-stable subsheaf $F$ with $\overline\mu_{H,D}(F)> \overline\mu_{H,D}(E)$.  Then the corresponding exact sequence of sheaves $$0\to F\to E\to G\to 0$$ is actually a sequence of objects in $\cT_\beta$.  Indeed, any quotient of an object in $\cT_\beta$ is in $\cT_\beta$, and $F\in \cT_\beta$ by construction.  Thus this is actually an exact sequence in $\cA_\beta$.  The formula for the Bridgeland slope then shows that $\mu_{\sigma_{\beta,\alpha}}(F) > \mu_{\sigma_{\beta,\alpha}}(E)$ for $\alpha \gg 0$, violating the $\sigma_{\beta,\alpha}$-semistability of $E$.  We conclude that $E$ is $\overline\mu_{H,D}$-semistable.  To see that $E$ is $(H,D)$-semistable, suppose there is a sequence $$0\to F\to E\to G\to 0$$ of sheaves of the same $\overline\mu_{H,D}$-slope, but $\overline\Delta_{H,D}(F)<\overline\Delta_{H,D}(E)$.  Then the formula for the Bridgeland slope gives $\mu_{\sigma_{\beta,\alpha}}(F)> \mu_{\sigma_{\beta,\alpha}}(E)$ for \emph{every} $\alpha$, again contradicting the $\sigma_{\beta,\alpha}$-semistability of $E$ for large $\alpha$.

(2) Suppose $E$ is $(H,D)$-semistable of character ${\bf v}$, and suppose $F^\bullet$ is a subobject of $E$ in $\cA_\beta$.  Taking the long exact sequence in cohomology sheaves of the exact sequence $$0\to F^\bullet \to E\to G^\bullet\to 0$$ in $\cA_\beta$ gives an exact sequence of sheaves $$0\to \rH^{-1}(F^\bullet)\to 0\to \rH^{-1}(G^{\bullet})\to \rH^0(F^{\bullet}) \to E\to H^0(G^{\bullet})\to 0.$$  Therefore $\rH^{-1}(F^{\bullet})=0$, i.e. $F:=F^{\bullet}$ is a sheaf in $\cT_\beta$.  The $(H,D)$-semistability of $E$ then gives $\overline \mu_{H,D}(F) \leq \overline \mu_{H,D}(E)$, with $\overline \Delta_{H,D}(F)\geq \overline \Delta_{H,D}(E)$ in case of equality.  The formula for $\mu_{\sigma_{\beta,\alpha}}$ then shows that if $\alpha \gg 0$ we have $\mu_{\sigma_{\beta,\alpha}}(F) \leq \mu_{\sigma_{\beta,\alpha}}(E)$.  It follows from the finiteness of the walls that $E$ is actually $\sigma_{\beta,\alpha}$-semistable for large $\alpha$.
\end{proof}

In particular, if ${\bf v}\in K_{\num}(X)$ is the character of an $(H,D)$-semistable sheaf of positive rank, then there is some largest wall $W$ lying to the left of the vertical wall (or, possibly, there are no walls left of the vertical wall).  This wall is called the \emph{Gieseker wall}.  For stability conditions $\sigma$ in the open chamber $\cC$ bounded by the Gieseker wall and the vertical wall, we have $$M_\sigma({\bf v}) \cong M_{H,D}({\bf v}).$$  Therefore, any moduli space of twisted semistable sheaves can be recovered as a moduli space of Bridgeland semistable objects.

\section{Examples on \texorpdfstring{$\P^2$}{the projective plane}}\label{sec-exP2} In this subsection we investigate a couple of the first interesting examples of Bridgeland stability conditions and their relationship to birational geometry.  We focus here on the characters of some small Hilbert schemes of points on $\P^2$.  In these cases the definitions simplify considerably, and things can be understood explicitly.  

\subsection{Notation} Let $X=\P^2$, and fix the standard polarization $H$.  We take $D=0$; in general, the choice of twisting divisor is only interesting modulo the polarization, as adding a multiple of the polarization to $D$ only translates the $(H,D)$-slice.  The twisting divisor becomes more relevant in examples of higher Picard rank.  Additionally, since $K_{\P^2}$ is parallel to $H$, we may as well work with the ordinary slope and discriminant $$\mu = \frac{\ch_1}{r} \qquad \Delta= \frac{1}{2}\mu^2 - \frac{\ch_2}{r}$$ instead of the more complicated $\overline \mu_{H,0}$ and $\overline \Delta_{H,0}$.  With these conventions, if ${\bf v}$ and ${\bf w}$ are characters of positive rank then the wall $W({\bf v},{\bf w})$ has center $(s_W,0)$ and radius $\rho_W$ given by \begin{align*}s_W &= \frac{\mu({\bf v}) + \mu({\bf w})}{2} - \frac{\Delta({\bf v})-\Delta({\bf w})}{\mu({\bf v})-\mu({\bf w})}\\
\rho_W^2 &= (s_W- \mu({\bf v}))^2-2\Delta({\bf v}).
\end{align*}  
If we further let ${\bf v} = \ch I_Z$ be the character of an ideal of a length $n$ scheme $Z\in \P^{2[n]}$, then the formulas further simplify to \begin{align*}
s_W &= \frac{\mu({\bf w})}{2} + \frac{n-\Delta({\bf w})}{\mu({\bf w})}\\ \rho_W^2 &= s_W^2-2n.
\end{align*}
The main question to keep in mind is the following.

\begin{question}
Let $I_Z$ be the ideal sheaf of $Z\in \P^{2[n]}$.  For which stability conditions $\sigma$ in the slice is $I_Z$ a $\sigma$-semistable object?  What does the destabilizing sequence of $I_Z$ look like along the wall where it is destabilized?
\end{question}

Note that since $I_Z$ is a Gieseker semistable sheaf, it is $\sigma_{\beta,\alpha}$-semistable if $\alpha \gg 0$ and $\beta<0 = \mu(I_Z)$.  There will be some wall $W$ left of the vertical wall where $I_Z$ is destabilized by some subobject $F$. For stability conditions $\sigma$ below this wall, $I_Z$ is never $\sigma$-semistable.  Thus the region in the slice where $I_Z$ is $\sigma$-semistable is bounded by the wall $W$ and the vertical wall.  It potentially consists of several of the chambers in the wall-and-chamber decomposition of the slice.

\subsection{Types of walls} There are two very different ways in which an ideal sheaf $I_Z$ of length $n$ can be destabilized along a wall.  The simplest way $I_Z$ can be destabilized is if it is destabilized by an actual subsheaf, i.e. if there is an exact  sequence of sheaves $$0\to I_Y(-k)\to I_Z \to T\to 0$$ giving rise to the wall for some zero-dimensional scheme $Y$ of length $\ell$.  The character ${\bf w} = \ch I_W(-k)$ has $(r,\mu,\Delta) = (1,-k,\ell)$, so this wall has center $(s_W,0)$ with \begin{equation}\label{rank1center} s_W=-\frac{k}{2} -\frac{n-\ell}{k}.\end{equation}
A wall obtained in this way is called a \emph{rank one wall}.

On the other hand, subobjects of $I_Z$ in the categories $\cA_\beta$ need not be subsheaves of $I_Z$!  In particular, it is entirely possible that $I_Z$ is destabilized by a sequence $$0\to F\to I_Z\to G \to 0$$ where ${\bf w} = \ch F$ has $\rk{\bf w}\geq 2$.  Such destabilizing sequences, giving so-called \emph{higher rank walls}, are somewhat more troublesome to deal with.  It will be helpful to bound their size, which we now do. 

As in the proof of Theorem \ref{thm-largeVol}, the long exact sequence of cohomology sheaves shows that any subobject $F\subset I_Z$ in a category $\cA_\beta$ must actually be a sheaf (but not necessarily a subsheaf).  Let $K$ and $C$ be the kernel and cokernel, respectively, of the map of sheaves $F\to I_Z$, so that there is an exact sequence of sheaves  $$0\to K\to F\to I_Z\to C\to 0.$$ In order for $G$ to be in the categories $\cA_\beta$ along the wall $W = W({\bf w},{\bf v})$ (which must be the case by Proposition \ref{prop-destSeq}), it is necessary and sufficient that we have $K\in \cF_\beta$ and $C\in \cT_\beta$ for all $\beta$ along the wall.  Indeed, $K$ and $C$ are the cohomology sheaves of the mapping cone of the map $F\to I_Z$, so this follows from Exercise \ref{ex-exact}.  The sequence $$0\to F\to I_Z\to G\to 0$$ will be exact in the categories along the wall if $F$ is additionally in $\cT_\beta$ for $\beta$ along the wall.  These basic considerations lead to the following result.

\begin{lemma}[\cite{ABCH}, or see {\cite[Lemma 3.1 and Corollary 3.2]{Boetal}} for a generalization to arbitrary surfaces]\label{lem-higherRank}
If an ideal sheaf $I_Z$ of $n$ points in $\P^2$ is destabilized along a wall $W$ given by a subobject $F$ of rank at least $2$, then the radius $\rho_W$ of $W$ satisfies $$\rho_W^2 \leq \frac{n}{4}.$$
\end{lemma}
\begin{proof}
We use the notation from above.  Since $I_Z$ is rank $1$ and torsion-free, a nonzero map $F\to I_Z$ has torsion cokernel.  Therefore $C$ is torsion, and it is no condition at all to have $C\in \cT_\beta$ along the wall.  We further deduce that $c_1(C)\geq 0$, so $c_1(K) \geq c_1(F)$ and $\rk(K) = \rk(F)-1$.  Let $(s_W,0)$ and $\rho_W$ be the center and radius of $W$.  Since $F\in \cT_\beta$ along the wall and $K\in \cF_\beta$ along the wall, we have  $$2\rho_W \leq \mu(F)-\mu(K) = \frac{c_1(F)}{\rk(F)}-\frac{c_1(K)}{\rk(K)}\leq - \frac{c_1(F)}{\rk(F)(\rk(F)-1)}\leq - \mu(F) \leq -s_W-\rho_W,$$ so $3\rho_W \leq -s_W$.  Squaring both sides, $9\rho_W^2 \leq s_W^2 = \rho_W^2+2n$ by the formula for the radius.  The result follows.
\end{proof}

\subsection{Small examples}  We now consider the stability of ideal sheaves of small numbers of points in $\P^2$ in detail.

\begin{example}[Ideals of 2 points]\label{ex-2ptBridgeland}
Let $I_Z$ be the ideal of a length $2$ scheme $Z\in \P^{2[2]}$.  Such an ideal fits in an exact sequence $$0\to \OO_{\P^2}(-1)\to I_Z\to \OO_L(-2)\to 0$$ where $L$ is the line spanned by $Z$.  If $W = W(\ch \OO_{\P^2}(-1),I_Z)$ is the wall defined by this sequence, then $Z$ is certainly not $\sigma$-semistable for stability conditions $\sigma$ inside $W$.  On the other hand, we claim that $I_Z$ is $\sigma$-semistable for stability conditions $\sigma$ on or above $W$.

To see this, we rule out the possibility that $I_Z$ is destabilized along some wall $W'$ which is larger than $W$.  The wall $W$ has center $(s_W,0)$ with $s_W = -5/2$ by Equation \ref{rank1center}.  Its radius is $\rho_W = 3/2$, so the wall $W$ passes through the point $(-1,0)$. If $W'$ is given by a rank $1$ subobject $I_Y(-k)$ then we must have $-k > -1$ in order for $I_Y(-k)$ to be in the categories $\cA_\beta$ along the wall $W'$.  This then forces $k=0$, which means $I_Y$ does not define a semicircular wall. This is absurd.

The other possibility is that $W'$ is a higher-rank wall.  But then by Lemma \ref{lem-higherRank}, $W'$ has radius $\rho_{W'}$ satisfying $\rho_{W'}^2 \leq 1/2$.  This contradicts that $W'$ is larger than $W$.
\end{example} 

Note that in the above example, if $\sigma$ is a stability condition on the wall $W$ then $I_Z$ is strictly $\sigma$-semistable and $S$-equivalent to any ideal $I_{Z'}$ where $Z'$ lies on the line spanned by $Z$.  Thus the set of $S$-equivalence classes of $\sigma$-semistable objects is naturally identified with $\P^{2*}$.  

\begin{example}[Ideals of $3$ collinear points]\label{ex-3ptBridgelandCollinear}
Let $I_Z$ be the ideal of a length $3$ scheme $Z\in \P^{2[3]}$ which is supported on a line.  As in Example \ref{ex-2ptBridgeland}, we claim that $I_Z$ is destablized by the sequence $$0\to \OO_{\P^2}(-1)\to I_Z\to \OO_L(-3)\to 0.$$  That is, if $W$ is the wall corresponding to the sequence, then $I_Z$ is $\sigma$-semistable for conditions $\sigma$ on or above the wall.  (From the existence of the sequence it is immediately clear that $I_Z$ is not $\sigma$-semistable below the wall.)

We compute $s_W = -7/2$ and $\rho_W = \frac{5}{2}$.  As in Example \ref{ex-2ptBridgeland}, we conclude that there is no larger rank $1$ wall.  Any higher rank wall  $W'$ would have $\rho_{W'}^2 \leq 3/2$, so there can be no larger higher rank wall either.  Therefore $I_Z$ is $\sigma$-semistable on and above $W$.
\end{example}

For the next example we will need one additional useful fact.

\begin{proposition}[{\cite[Proposition 6.2]{ABCH}}]\label{prop-lineBundle}
A line bundle $\OO_{\P^2}(-k)$ or a shifted line bundle $\OO_{\P^2}(-k)[1]$ is $\sigma_{\beta,\alpha}$-stable whenever it is in the category $\cA_\beta$.  Thus, $\OO_{\P^2}(-k)$ is $\sigma_{\beta,\alpha}$-stable if  $\beta < -k$, and $\OO_{\P^2}(-k)[1]$ is $\sigma_{\beta,\alpha}$-stable if $\beta \geq -k$.
\end{proposition}

In the next example we see our first example of an ideal sheaf destabilized by a higher rank subobject.

\begin{example}[Ideals of $3$ general points]\label{ex-3ptBridgeland} Let $I_Z$ be the ideal of a length $3$ scheme $Z\in \P^{2[3]}$ which is \emph{not} supported on a line.  In this case, the ideal $I_Z$ has a minimal resolution of the form $$0\to \OO_{\P^2}(-3)^2\to \OO_{\P^2}(-2)^3\to I_Z\to 0$$ or, equivalently, there is a distinguished triangle $$\OO_{\P^2}(-2)^3\to I_Z\to \OO_{\P^2}(-3)^2[1] \to \cdot.$$ Consider the wall $W = W(\OO_{\P^2}(-2),I_Z)$ defined by this sequence.  It has center at $(s_W,0)$ with $s_W = -5/2$, and its radius is $1/2$.  By Proposition \ref{prop-lineBundle} and Exercise \ref{ex-exact}, the above triangle gives an exact sequence $$0\to \OO_{\P^2}(-2)^3\to I_Z\to \OO_{\P^2}(-3)^2[1]\to 0$$ in the categories $\cA_\beta$ along the wall.  Then for any $\sigma$ on the wall, $I_Z$ is an extension of $\sigma$-semistable objects of the same slope, and hence is $\sigma$-semistable.  It follows that $I_Z$ is destabilized precisely along $W$.
\end{example}

\begin{remark}[Correspondence between birational geometry and Bridgeland stability] In \cite[\S 10]{ABCH}, this chain of examples is continued in great detail.  The regions of $\sigma$-semistability  of ideal sheaves $I_Z$ of up to $9$ points are completely determined by similar methods.  A remarkable correspondence between these regions of stability and the stable base locus decomposition was observed and conjectured to hold in general.  The following result has since been proved by Li and Zhao.

\begin{theorem}[\cite{LiZhao}]
Let $Z\in \P^{2[n]}$.  Let $W$ be the Bridgeland wall where the ideal sheaf $I_Z$ is destabilized.  Also, let $yH-\frac{1}{2}B$ be the ray in the Mori cone past which the point $Z\in \P^{2[n]}$ enters the stable base locus.  Then $$s_W = -y-\frac{3}{2}.$$
\end{theorem}

Therefore, computations in Bridgeland stability provide a dictionary between semistability and birational geometry.  Compare with Examples \ref{ex-P2Hilb2}, \ref{ex-P2Hilb3}, \ref{ex-2ptBridgeland}, \ref{ex-3ptBridgelandCollinear}, \ref{ex-3ptBridgeland}, which establish the cases $n=2,3$ of the result.

More conceptually, Li and Zhao prove that the alternate birational models of any moduli space of sheaves on $\P^2$ can be interpreted as a Bridgeland moduli space, and they match up the walls in the Mori chamber decomposition of the effective cone with the walls in the wall-and-chamber decomposition of the stability manifold.    As a consequence, they are able to give new computations of the effective, movable, and ample cones of divisors on these spaces.  A crucial ingredient in this program is the smoothness of these Bridgeland moduli spaces, as well as a Dr\'ezet-Le Potier type classification of characters ${\bf v}$ for which Bridgeland moduli spaces are nonempty \cite[Theorems 0.1 and 0.2]{LiZhao}.  
\end{remark}

The next exercise computes the Gieseker wall for a Hilbert scheme of points on $\P^2$.  This is the easiest case of the main problem we will discuss in the next section.

\begin{exercise}
Following Examples \ref{ex-2ptBridgeland} and \ref{ex-3ptBridgelandCollinear}, show that the largest wall where some ideal sheaf $I_Z$ of $n$ points is destabilized is the wall $W(\ch \OO_{\P^2}(-1),I_Z)$.  Furthermore, an ideal $I_Z$ is destabilized along this wall if and only if $Z$ lies on a line.
\end{exercise}

\begin{remark}
A similar program to the above has also been undertaken on some other rational surfaces such as Hirzebruch and del Pezzo surfaces.  See \cite{BC}.
\end{remark}

\section{The positivity lemma and nef cones}\label{sec-positivity}

We close the survey by discussing the positivity lemma of Bayer and Macr\`i and recent applications of this tool to the computation of cones of nef divisors on Hilbert schemes of points and moduli spaces of sheaves.  This provides an example where Bridgeland stability provides insight that at present is not understood from a more classical point of view.

\subsection{The positivity lemma}  The positivity lemma is a tool for constructing nef divisors on moduli spaces of Bridgeland-semistable objects.  On a surface, (twisted) Gieseker moduli spaces can themselves be viewed as Bridgeland moduli spaces, so this will also allow us to construct nef divisors on classical moduli spaces.  As with the construction of divisors on Gieseker moduli spaces, the starting point is to define a divisor on the base of a family of objects.  When the moduli space carries a universal family, the family can be used to define a divisor on the moduli space.

In this direction, let $\sigma=(Z,\cA)$ be a stability condition on $X$, and let $\cE/S$ be a flat family of $\sigma$-semistable objects of character ${\bf v}$ parameterized by a proper algebraic space $S$.  We define a numerical divisor class $D_{\sigma,\cE}\in N^1(S)$ on $S$ depending on $\cE$ and $\sigma$ by specifying the intersection $D_{\sigma,\cE}.C$ with every curve class $C\subset S$.  Let $\Phi_\cE: D^b(S) \to D^b(X)$ be the Fourier-Mukai transform with kernel $\cE$, defined by $$\Phi_{\cE}(F) = q_*(p^*F \te \cE),$$ where $p:S\times X\to S$ and $q:S\times X\to X$ are the projections and all the functors are derived.  Then we declare $$D_{\sigma,\cE}.C = \Im \left(- \frac{Z(\Phi_\cE(\OO_C))}{Z({\bf v})}\right).$$ 

\begin{remark}
Note that if $Z({\bf v}) = -1$ then the formula becomes $$D_{\sigma,\cE}.C = \Im(Z(\Phi_\cE(\OO_C))).$$ If $\Phi_{\cE}(\OO_C)\in \cA$, then $D_{\sigma,\cE}.C\geq 0$ would follow from the positivity of the central charge.  While it is not necessarily true that $\Phi_{\cE}(\OO_C)\in \cA$, this fact nonetheless plays an important role in the proof of the positivity lemma.
\end{remark}

The positivity lemma states that this assignment actually defines a nef divisor on $S$.  Furthermore, there is a simple criterion to detect the curves $C$ meeting the divisor orthogonally.

\begin{theorem}[Positivity lemma, Theorem 4.1 \cite{BM2}]\label{thm-positivity}
The above assignment defines a well-defined numerical divisor class $D_{\sigma,\cE}$ on $S$.  This divisor is nef, and a complete, integral curve $C\subset S$ satisfies $D_{\sigma,\cE}. C = 0$ if and only if the objects parameterized by two general points of $C$ are $S$-equivalent with respect to $\sigma$.
\end{theorem}

If the moduli space $M_{\sigma}({\bf v})$ carries a universal family $\cE$, then Theorem \ref{thm-positivity} constructs a nef divisor $D_{\sigma,\cE}$ on the moduli space.  In fact, the divisor does not depend on the choice of $\cE$; we will see this in the next subsection.

\begin{remark}
If multiplies of $D_{\sigma,\cE}$ define a morphism from $S$ to projective space, then the curves $C$ contracted by this morphism are characterized as the curves with $D_{\sigma,\cE}. C = 0$.  Thus, in a sense, all the interesting birational geometry coming from such a nef divisor $D_{\sigma,\cE}$ is due to $S$-equivalence.

Unfortunately, in general, a nef divisor does not necessarily give rise to a morphism---multiples of the divisor do not necessarily have any sections at all.  However, in such cases the positivity lemma is especially interesting.  Indeed, one of the easiest ways to construct nef divisors is to pull back ample divisors by a morphism (recall Examples \ref{ex-nefpullback} and \ref{ex-nefCompute}).  The positivity lemma can potentially produce nef divisors not corresponding to any any map at all, in which case nefness is classically more difficult to check.
\end{remark}

\subsection{Computation of divisors} 
It is interesting to relate the Bayer-Macr\`i divisors $D_{\sigma,\cE}$ with the determinantal divisors on a base $S$ arising from a family $\cE/S$.  Now would be a good time to review \S \ref{ssec-lineBundles}.  Recall that the Donaldson homomorphism is a map $$\lambda_{\cE} : {\bf v}^\perp\to N^1(S)$$ depending on a choice of family $\cE/S$, where ${\bf v}^\perp \subset K_{\num}(X)_\R$.  Reviewing the definition of $\lambda_{\cE}$, the definition only actually depends on the class of $\cE \in K^0(S\times X)$, so it immediately extends to the case where $\cE$ is a family of $\sigma$-semistable objects.  

Since the Euler pairing $(-,-)$ is nondegenerate on $K_{\num}(X)_\R$, any linear functional on $K_{\num}(X)_\R$ vanishing on ${\bf v}$ can be represented by a vector in ${\bf v}^\perp$.  In particular, there is a unique vector ${\bf w}_{Z}\in {\bf v}^\perp$ such that $$\Im\left(-\frac{Z({\bf w})}{Z({\bf v})}\right)=({\bf w}_{Z},{\bf w})$$ holds for all ${\bf w}\in K_{\num}(X)_\R$. Note that the definition of ${\bf w}_Z$ is essentially purely linear-algebraic, and makes no reference to $S$ or ${\cE}$.  
The next result shows that the Bayer-Macr\`i divisors are all determinantal.

\begin{proposition}[{\cite[Proposition 4.4]{BM2}}]
We have $$D_{\sigma,\cE} = \lambda_\cE({\bf w}_{Z}).$$
\end{proposition}

If $N$ is any line bundle on $S$, then we have $$D_{\sigma,\cE \te p^*N} = \lambda_{\cE \te p^*N}({\bf w}_Z) = \lambda_{\cE}({\bf w}_Z) = D_{\sigma,\cE}.$$ In particular, if $S$ is a moduli space $M_{\sigma}({\bf v})$ with a universal family $\cE$, then the divisor $D_{\sigma}:=D_{\sigma,\cE}$ does not depend on the choice of universal family.

\begin{remark}
See \cite[\S4]{BM2} for less restrictive hypotheses under which a divisor can be defined on the moduli space.
\end{remark}

In explicit cases, it can be useful to compute the character ${\bf w}_Z$ in more detail.  The next result does this in the case of an $(H,D)$-slice of divisors on a smooth surface $X$ (review \S \ref{ssec-surfaceConds}).

\begin{lemma}[{\cite[Proposition 3.8]{Boetal}}]
Let $X$ be a smooth surface and let $H,D\in \Pic(X)\te \R$, with $H$ ample.  If $\sigma$ is a stability condition in the $(H,D)$-slice with center $(s_W,0)$, then the character ${\bf w}_Z$ is a multiple of $$(-1,-\frac{1}{2}K_X+s_WH+D,m)\in {\bf v}^\perp,$$ where we write Chern characters as $(\ch_0,\ch_1,\ch_2)$.  Here the number $m$ is determined by the property that the character is in ${\bf v}^\perp$.
\end{lemma}

\subsection{Gieseker walls and nef cones} For the rest of the survey we let $X$ be a smooth surface and fix an $(H,D)$-slice $\Pi$ of stability conditions.  Let ${\bf v}\in K_{\num}({\bf v})$ be the Chern character of an $(H,D)$-semistable sheaf of positive rank.  Additionally assume for simplicity that $M_{H,D}({\bf v})$ has a universal family $\cE$, so that in particular every $(H,D)$-semistable sheaf is $(H,D)$-stable.  Recall that the \emph{Gieseker wall} $W$ for ${\bf v}$ in the $(H,D)$-slice is, by definition, the largest wall where an $(H,D)$-semistable sheaf of character ${\bf v}$ is destabilized.  For conditions $\sigma$ on or above $W$, every $(H,D)$-semistable sheaf is $\sigma$-semistable.  Therefore, for any such $\sigma$, the universal family $\cE$ is a family of $\sigma$-semistable objects parameterized by $M_{H,D}({\bf v})$.  Each condition $\sigma$ on or above the wall therefore gives a nef divisor $D_{\sigma}=D_{\sigma,\cE}$ on the moduli space.  

\begin{corollary}
With notation as above, if $s \leq s_W$, then the divisor on $M_{H,D}({\bf v})$ corresponding to the class $$(-1,-\frac{1}{2}K_X+sH+D,m)\in {\bf v}^\perp $$ under the Donaldson homomorphism is nef.
\end{corollary}

Now let $\sigma$ be a stability condition on the Gieseker wall.  It is natural to wonder whether the ``final'' nef divisor $D_\sigma$ produced by this method is a boundary nef divisor.  This may or may not be the case.  By Theorem \ref{thm-positivity}, the divisor $D_\sigma$ is on the boundary of $\Nef(M_{H,D}({\bf v}))$ if and only if there is a curve in $M_{H,D}({\bf v})$ parameterizing sheaves which are generically $S$-equivalent with respect to the stability condition $\sigma$.  This happens if there is some sheaf $E\in M_{H,D}({\bf v})$ destabilized along $W$ by a sequence $$0\to F\to E\to G\to 0$$ where it is possible to vary the extension class in $\Ext^1(G,F)$ to obtain non-isomorphic objects $E'$.  This can be subtle, and typically requires further analysis.

\subsection{Nef cones of Hilbert schemes of points on surfaces}  In this section we survey the recent results of \cite{Boetal} computing nef divisors on the Hilbert scheme $X^{[n]}$ of points on a smooth surface of irregularity $q(X) = 0$.  Let ${\bf v} = \ch I_Z$, where $Z\in X^{[n]}$.  For each pair of divisors $(H,D)$ on $X$, we can interpret $X^{[n]}$ as the moduli space $M_{H,D}({\bf v})$.  A stability condition $\sigma$ in the $(H,D)$-slice on a wall $W$ with center $(s_W,0)$ induces a divisor $D_\sigma$ on $X^{[n]}$ with class a multiple of $$\frac{1}{2}K_X^{[n]} - s_WH^{[n]}-D^{[n]}-\frac{1}{2}B.$$  The ray spanned by this class tends to the ray spanned by $H^{[n]}$ as $s_W\to -\infty$. As $s_W$ varies in the above expression we obtain a two-dimensional cone of divisors in $N^1(X^{[n]})$ containing the ray spanned by the nef divisor $H^{[n]}$.  The positivity lemma allows us to study the nefness of divisors in this cone by studying the Gieseker wall for ${\bf v}$ in the $(H,D)$-slice.  Changing the twisting divisor $D$ changes which two-dimensional cone we look at, and the entire nef cone of $X^{[n]}$ can be studied by the systematic variation of the twisting divisor. 

The main result we discuss in this section addresses the computation of the Gieseker wall in an $(H,D)$-slice, at least assuming the number $n$ of points is sufficiently large.  We find that the Gieseker wall, or more precisely the subobject computing it, ``stabilizes'' once $n$ is sufficiently large.
 
\begin{theorem}[{\cite[various results from \S 3]{Boetal}}]\label{thm-hilbAsymptotic}
There is a curve $C\subset X$ (depending on $H,D$) such that if $n\gg 0$ then the Gieseker wall for ${\bf v}$ in the $(H,D)$-slice is computed by the rank 1 subobject $\OO_X(-C)$.  The intersection number $C.H$ is minimal among all effective curves $C$ on $X$.  The divisor $D_\sigma$ corresponding to a stability condition $\sigma$ on the Gieseker wall is an extremal nef divisor.  Orthogonal curves to $D_{\sigma}$ can be obtained by letting $n$ points move in a $g_n^1$ on $C$.
\end{theorem}

Note that everything here has already been verified for $X=\P^2$, and in fact $n\geq 2$ is sufficient in this case.  The destabilizing subobject is always $\OO_{\P^2}(-1)$.

\begin{proof}[Sketch proof]
Consider the character ${\bf v}$ as varying with $n$.  Then $\overline\mu_{H,D}({\bf v})$ is constant, and $\overline \Delta_{H,D}({\bf v})$ is of the form $n + {\mathrm {const}}$.  Consider the wall $W'$ given by a rank 1 object $I_Y(-C)$ with $C$ an effective curve, and put ${\bf w} = \ch I_Y(-C)$.  The wall $W'$ has center at $(s_{W'},0)$ with $$s_{W'} = \frac{\overline \mu_{H,D}({\bf v}) + \overline \mu_{H,D}({\bf w})}{2} - \frac{\overline \Delta_{H,D}({\bf v}) - \overline \Delta_{H,D}({\bf w})}{\overline \mu_{H,D}({\bf v}) - \overline \mu_{H,D}({\bf w})}.$$ As a function of $n$, this looks like \begin{equation}\label{centerW'} s_{W'} = -\frac{n}{\mu_H({\bf v})-\mu_H({\bf w})}+\mathrm{const} = \frac{n}{ \mu_H({\bf w})} + \mathrm{const} = -\frac{n}{C.H}+\mathrm{const},\end{equation} where the constant depends on ${\bf w}$.  Correspondingly, the radius $\rho_{W'}$ grows approximately linearly in $n$.

Note that the numerical wall given by $\OO_X(-C)$ is always at least as large as the numerical wall given by $I_Y(-C)$, by a discriminant calculation.  Furthermore, if $I_Y(-C)$ gives an actual wall, i.e. if there is some $I_Z\in X^{[n]}$ fitting in a sequence $$0\to I_Y(-C) \to I_Z \to T\to 0,$$ then $\OO_X(-C)$ also gives an actual wall.  Thus, if the Gieseker wall is computed by a rank $1$ sheaf then it is computed by a line bundle $\OO_X(-C)$.

In fact, for $n\gg 0$ the Gieseker wall is computed by a line bundle $\OO_X(-C)$ and not by some higher rank subobject.  This is because an analog of Lemma \ref{lem-higherRank} for arbitrary surfaces shows that any higher rank wall for ${\bf v}$ in the $(H,D)$-slice has radius \emph{squared} bounded by $n$ times a constant depending on $H,D$.  On the other hand, as soon as we know there is \emph{some} wall given by a rank $1$ subobject it follows that there are walls with radius which is linear in $n$, implying that the Gieseker wall is not a higher rank wall.

To see that there is some rank $1$ wall if $n \gg 0$, let $C$ be any effective curve.  For some $Z\in X^{[n]}$, there is an exact sequence of sheaves $$0\to \OO_X(-C) \to I_Z\to I_{Z\subset C}\to 0.$$ We know the numerical wall $W'$ corresponding to the subobject $\OO_X(-C)$ has radius which grows linearly with $n$.  In particular, for $n \gg 0$ the wall is nonempty.  Furthermore, since $\overline\mu_{H,D}(\OO_X(-C))$ and $\overline\mu_{H,D}(I_Z)$ are constant with $n$ but $\overline\Delta_{H,D}(I_Z)$ is unbounded with $n$, the sheaf $\OO_X(-C)$ is eventually in some of the categories along the wall $W'$.  Thus the above exact sequence of sheaves is an exact sequence along the wall, and this wall is larger than any higher rank wall.  We conclude that $I_Z$ is either destabilized along $W'$ or destabilized along some possibly larger rank $1$ wall.  Either way, there is a rank $1$ wall, and the Gieseker wall is a rank $1$ wall, computed by some line bundle $\OO_X(-C)$ with $C$ effective.

More precisely, the curve $C$ such that the subsheaf $\OO_X(-C)$ computes the Gieseker wall for $n\gg 0$ is the effective curve which gives the largest numerical (and hence actual) wall.  Considering the Formula (\ref{centerW'}) for the center of the wall determined by $\OO_X(-C)$, we find that $C$ must be an effective curve of minimal $H$-degree.  Furthermore, $C$ must be chosen to minimize the constant which appears in that formula (this depends additionally on $D$).  Any such curve $C$ which asymptotically minimizes Formula (\ref{centerW'}) in this way computes the Gieseker wall for $n\gg 0$.  Curves orthogonal to the divisor $D_\sigma$ given by a stability condition on the Gieseker wall can now be obtained by varying the extension class in the sequence $$0\to \OO_X(-C) \to I_Z\to I_{Z\subset C}\to 0;$$ this corresponds to letting $Z$ move in a pencil on $C$, which can certainly be done for $n\gg 0$.
\end{proof}

More care is taken in \cite{Boetal} to determine the precise bounds on $n$ which are necessary for the various steps of the proof.  The general method is applied to compute nef cones of Hilbert schemes of sufficiently many points on very general surfaces in $\P^3$, very general double covers of $\P^2$, and del Pezzo surfaces of degree $1$.  The last example provides an example of a surface of higher Picard rank, where the variation of the twisting divisor is exploited.  See \cite[\S 4-5]{Boetal} for details.  We highlight one of the first interesting cases where the answer appears to be unknown.

\begin{problem}
Let $X\subset \P^3$ be a very general quintic surface, so that the Picard rank is $1$ by the Noether-Lefschetz theorem.  Compute the nef cone of $X^{[2]}$ and $X^{[3]}$.
\end{problem}

Once $n\geq 4$ in the previous example, the nef cone is known by the general methods above.  See \cite[Proposition 4.5]{Boetal}.

\subsection{Nef cones of moduli spaces of sheaves on surfaces}
We close our discussion with a survey of the main result of \cite{CHNefGeneral} on the cone of nef divisors on a moduli space of sheaves with large discriminant on an arbitrary smooth surface.  In the case of $\P^2$, this result was first discovered in the papers \cite{CHAmple,LiZhao}.  The picture for an arbitrary surface is a modest simultaneous generalization of the $\P^2$ case as well as the Hilbert scheme case for an arbitrary surface (see \S 7.4 or \cite{Boetal}).  

 Again let $X$ be a smooth surface and let $H,D$ be divisors giving a slice of stability conditions.  Let ${\bf v}$ be the character of an $(H,D)$-semistable sheaf of positive rank.  We assume the discriminant $\overline\Delta_{H,D} ({\bf v})\gg 0$ is sufficiently large.  Suppose the moduli space $M_{H,D}({\bf v})$ carries a (quasi-)universal family.  The goal of \cite{CHNefGeneral} is to compute the Gieseker wall for ${\bf v}$ in the $(H,D)$-slice and to show that the divisor $D_\sigma$ corresponding to a stability condition $\sigma$ on the Gieseker wall is a boundary nef divisor.

The basic picture is similar to the case of a Hilbert scheme of points, and indeed Theorem \ref{thm-hilbAsymptotic} will follow as a special case of this more general result.  However, the asymptotics can easily be made much more explicit in the Hilbert scheme case.  The common thread between the two results is that as the discriminant $\overline\Delta_{H,D}({\bf v})$ is increased, the character ${\bf w}$ of a destabilizing subobject giving rise to the Gieseker wall stabilizes.  It is furthermore easy to give properties which almost uniquely define the character ${\bf w}$.

\begin{definition}
Fix an $(H,D)$-slice.  An \emph{extremal Chern character} ${\bf w}$ for ${\bf v}$ is any character satisfying the following defining properties.
\begin{enumerate}[label=(E\arabic*)]
\item \label{cond-rankBound}We have $0<r({\bf w})\leq r({\bf v})$, and if $r({\bf w})=r({\bf v})$, then $c_1({\bf v})-c_1({\bf w})$ is effective.
\item \label{cond-slopeClose} We have $\mu_H({\bf w}) < \mu_H({\bf v})$, and $\mu_H({\bf w})$ is as close to $\mu_H({\bf v})$ as possible subject to \ref{cond-rankBound}.
\item \label{cond-stable} The moduli space $M_{H,D}({\bf w})$ is nonempty.
\item \label{cond-discriminantMinimal} The discriminant $\overline\Delta_{H,D}({\bf w})$ is as small as possible, subject to \ref{cond-rankBound}-\ref{cond-stable}.
\item \label{cond-rankMaximal} The rank $r({\bf w})$ is as large as possible, subject to \ref{cond-rankBound}-\ref{cond-discriminantMinimal}.
\end{enumerate}
\end{definition}

Note that properties \ref{cond-rankBound}-\ref{cond-discriminantMinimal} uniquely determine the slope $\mu_H({\bf w})$ and discriminant $\overline\Delta_{H,D}({\bf w})$, although $c_1({\bf w})$ is not necessarily uniquely determined.  Condition \ref{cond-rankMaximal} uniquely specifies the rank of ${\bf w}$.  We then have the following theorem.  Furthermore, notice that the definition does not depend on the discriminant $\overline\Delta_{H,D}({\bf v})$, so that ${\bf w}$ can be held constant as $\overline \Delta_{H,D}({\bf v})$ varies.

\begin{theorem}[\cite{CHNefGeneral}]\label{thm-moduliAsymptotic}
Suppose $\overline\Delta_{H,D}({\bf v}) \gg 0$. Then the Gieseker wall for ${\bf v}$ in the $(H,D)$-slice is computed by a destabilizing subobject of character ${\bf w}$, where ${\bf w}$ is an extremal Chern character for ${\bf v}$.  Furthermore, the divisor $D_\sigma$ corresponding to a stability condition $\sigma$ on the Gieseker wall is a boundary nef divisor.
\end{theorem}
 
The argument is largely similar to the proof of Theorem \ref{thm-hilbAsymptotic}.  First one shows that the destabilizing subobject along the Gieseker wall must actually be a subsheaf, and not some higher rank object.  This justifies restriction \ref{cond-rankBound} in the definition of ${\bf w}$ (note that if $r({\bf w}) = r({\bf v})$ then the only way there can be an injection of sheaves $F\to E$ with $\ch F = {\bf w}$ and $\ch E = {\bf v}$ is if the induced map $\det F\to \det E$ is injective, forcing $c_1({\bf v}) - c_1({\bf w})$ to be effective.

Next, one shows that the subsheaf defining the Gieseker wall must actually be an $(H,D)$-semistable sheaf.  Recalling the formula $$s_W = \frac{\overline\mu_{H,D}({\bf v})+\overline\mu_{H,D}({\bf w})}{2} - \frac{\overline\Delta_{H,D}({\bf v})-\overline\Delta_{H,D}({\bf w})}{\overline \mu_{H,D}({\bf v}) - \overline\mu_{H,D}({\bf w})}$$ for the center of a wall, conditions \ref{cond-slopeClose}-\ref{cond-discriminantMinimal} then ensure that the numerical wall defined by ${\bf w}$ is as large as possible when $\OV \Delta_{H,D}({\bf v})\gg 0$.  Therefore, the Gieseker wall for ${\bf v}$ is no larger than the wall defined by the extremal character ${\bf w}$.
  
\begin{remark}
Actually computing the extremal character ${\bf w}$ can be extremely challenging.  Minimizing the discriminant of ${\bf w}$ subject to the condition that the moduli space $M_{H,D}({\bf w})$ is nonempty essentially requires knowing the sharpest possible Bogomolov inequalities for semistable sheaves on $X$.  Conversely, if the nef cones of moduli spaces of sheaves on $X$ are known, strong Bogomolov-type inequalities can be deduced.  On surfaces such as $\P^2$ and $K3$ surfaces, the extremal character can be computed mechanically using the classification of semistable sheaves; recall for example \S\ref{sssec-existP2} and \S\ref{sssec-K3exist}.  
\end{remark}  

The proof of Theorem \ref{thm-moduliAsymptotic} diverges from the Hilbert scheme case when we need to show that the numerical wall for ${\bf v}$ defined by an extremal character ${\bf w}$ is an actual wall.  In the Hilbert scheme case, it is trivial to produce ideal sheaves $I_Z$ which are destabilized by a rank $1$ object $\OO_X(-C)$: we simply put $Z$ on $C$, and get an exact sequence $$0\to \OO_X(-C)\to I_Z\to I_{Z\subset C}\to 0$$ which is an exact sequence in the categories along the wall if the number of points is sufficiently large.  

To prove Theorem \ref{thm-moduliAsymptotic}, we instead need to produce $(H,D)$-semistable sheaves $E$ of character ${\bf v}$ fitting in sequences of the form $$0\to F \to E \to G\to 0$$ where $F$ is $(H,D)$-semistable of character ${\bf w}$.  This is somewhat technical.  Let ${\bf u} = \ch G$; then ${\bf u}$ has $r({\bf u})<r({\bf v})$, and $\overline \Delta_{H,D}({\bf u}) \gg 0$.  Therefore, by induction on the rank, we may assume the Gieseker wall of ${\bf u}$ has been computed.  We then show that the Gieseker walls for ${\bf w}$ and ${\bf u}$ are nested inside $W:=W({\bf v},{\bf w})$ if $\overline \Delta_{H,D}({\bf v})\gg 0$.  Therefore, any sheaves $F\in M_{H,D}({\bf w})$ and $G\in M_{H,D}({\bf u})$ are actually $\sigma$-semistable for any stability condition $\sigma$ on $W$.  Then any extension $E$ of $G$ by $F$ is $\sigma$-semistable, and it can further be shown that a general such extension is actually $(H,D)$-stable.  By varying the extension class, we can produce curves in $M_{H,D}({\bf v})$ parameterizing non-isomorphic $(H,D)$-stable sheaves; these curves are orthogonal to the nef divisor given by the Gieseker wall.  See \cite[\S 5-6]{CHNefGeneral} for details.

\begin{remark}
Several applications of Theorem \ref{thm-moduliAsymptotic} to simple surfaces are given in \cite[\S 7]{CHNefGeneral}.
\end{remark}

\bibliographystyle{plain}

\end{document}